\documentclass[a4paper,10pt]{amsart}
\usepackage[mathscr]{eucal}
\usepackage{amssymb,latexsym,bbm,amsmath,amsfonts}

\newtheorem{theorem}{Theorem}
\newtheorem{lemma}{Lemma}

\theoremstyle{definition}

\theoremstyle{remark}

\numberwithin{equation}{section}

\renewcommand{\Re}{\text{Re}\,}

\newcommand{\abs}[1]{\left\lvert #1 \right\rvert}
\newcommand{\Cov}{\operatorname{Cov}}
\newcommand{\Corr}{\operatorname{Corr}}

\begin{document}

\title[Running maxima of affine SNFDEs]
{On the almost sure running maxima of solutions of affine neutral
stochastic functional differential equations}
\author{John A. D. Appleby}
\address{
Edgeworth Centre for Financial Mathematics, School of Mathematical
Sciences, Dublin City University, Glasnevin, Dublin 9, Ireland}
\email{john.appleby@dcu.ie} \urladdr{webpages.dcu.ie/\textasciitilde
applebyj}
\author{Huizhong Appleby--Wu}
\address{Department of Mathematics, St Patrick's College, Drumcondra, Dublin 9, Ireland} \email{huizhong.applebywu@spd.dcu.ie}
\author{Xuerong Mao}
\address{Department of Mathematics and Statistics, University of Strathclyde,
Livingstone Tower, 26 Richmond Street, Glasgow G1 1XT, United
Kingdom.} \email{xuerong@stams.strath.ac.uk}
\urladdr{http://www.stams.strath.ac.uk/\textasciitilde xuerong}

\thanks{We gratefully acknowledge the support of this work by Science Foundation Ireland (SFI) under the
Research Frontiers Programme grant RFP/MAT/0018 ``Stochastic
Functional Differential Equations with Long Memory''. JA also thanks
SFI for the support of this research under the Mathematics
Initiative 2007 grant 07/MI/008 ``Edgeworth Centre for Financial
Mathematics''.} 
\date{11 July 2013}

\begin{abstract}
This paper studies the large fluctuations of solutions of
finite--dimensional affine stochastic neutral functional
differential equations with finite memory, as well as related
nonlinear equations. We find conditions under which the exact almost
sure growth rate of the running maximum of each component of the
system can be determined, both for affine and nonlinear equations.
The proofs exploit the fact that an exponentially decaying
fundamental solution of the underlying deterministic equation is
sufficient to ensure that the solution of the affine equation
converges to a stationary Gaussian process.
\end{abstract}

\maketitle


\section{Introduction}

In the last decade, a large number of papers have been written about the \emph{stability} of solutions
of stochastic neutral functional differential equations (SNFDEs).
Asymptotic (usually exponential) stability has been studied by Mao~\cite{Mao:1995,Mao97neut,Mao:2000}, Liao and Mao \cite{LiaoMao:1996}, Liu and Xia~\cite{Liu&Xia99}, Luo~\cite{Luo:2007}, Luo, Mao and Shen~\cite{LuoMaoShen:2006}, Shen and Liao~\cite{ShenLiao:1999}
and Jankovic, Randjelovic, and Jovanovic~\cite{JankRandJov:2009,RandJank:2007}. Equations with Markovian switching have also been studied in Mao, Shen and Yuan~\cite{MaoShenYuan}. Monographs which consider at least in part the theory of neutral SFDEs (including stability theory) have been written by Kolmanovskii, Nosov and Myskhis~\cite{KolNosov:1981,KolNosov:1986,KolMys:99} and Mao~\cite{Mao97,Mao2008}. The last book in particular contains results on asymptotic behaviour which do not necessarily appear in journals. The stability theory has even been extended to stochastic neutral partial equations; in this regard, we refer the reader to  Luo~\cite{Luo:2009} and Govindan~\cite{Govindin:2005}, for example.

Despite this surge in activity, it appear that much less work has been done on determining the asymptotic behaviour of SNFDEs whose solutions are not asymptotically stable. In part this stems from the
interest in neutral equations in control engineering, in which pathwise or moment stability is of
great importance. However, there are good reasons, both in terms of mathematical interest, and applications
to consider SNFDEs whose solutions are not asymptotically stable.
One paper in this direction, which considers stability in distribution is Frank~\cite{Frank:2005a}, which
studies conditions under which affine stochastic  neutral delay differential equations possess unique
solutions. Since the solution of such an equation is Gaussian, and the limiting distribution is stationary, it seems that the solution cannot be bounded. In a finite dimensional setting therefore, we might expect
solutions to obey
\[
\lim_{t\to\infty} \max_{0\leq s\leq t} |X(s)|=\infty, \quad\text{a.s.}
\]
The scalar process $t\mapsto X^\ast(t):=\max_{0\leq s\leq t} |X(s)|$ is called the \emph{running maximum}.

Therefore, it is natural to ask, at what rate does the running maxima tend to infinity, or, more precisely to find a deterministic
function $\rho$ with $\rho(t)\to\infty$ as $t\to\infty$ such that
\begin{equation} \label{eq.essgrowth}
\lim_{t\to\infty} \frac{X^\ast(t)}{\rho(t)}=1, \quad\text{a.s.}
\end{equation}
We call such a function $\rho$ the \emph{essential growth rate of the running maxima} of $X$. In applications this is important, as the size of the large fluctuations may represent the largest bubble or crash in a
financial market, the largest epidemic in a disease model, or a
population explosion in an ecological model.

To date there is comparatively little literature regarding the size of such large
fluctuations for SNFDEs, and to the best of our knowledge, no comprehensive
theory for affine stochastic neutral functional differential equations. In this paper, we determine
the essential growth rate of the running maximum for affine SNFDEs. The class of equations covered includes
equations with both point and distributed delay by using measures in the delay.
The results exploit the fact that given an exponentially decaying differential resolvent, the finite delay in the equation forces the limiting autocovariance function to decay exponentially fast, so that the
solution of the linear equation is an asymptotically stationary
Gaussian process. The results apply to both scalar and finite--dimensional
equations and can moreover be extended to equations with a weak nonlinearity at infinity.

The paper bears many similarities to results proved in a recent paper of the authors~\cite{AppleWu08linearsfde} which considers the large fluctuations of affine non--neutral stochastic functional differential equations.
Indeed the main results here are all analogues of those in~\cite{AppleWu08linearsfde}. However,
the proofs of both main results differ because the differential resolvent of the neutral equation
is not guaranteed to be differentiable, while that of the non--neutral functional differential equation
is differentiable. In the proofs of the main results in~\cite{AppleWu08linearsfde}, this differentiability plays a crucial role in controlling the behaviour of the process between discrete
mesh points at which the process is sampled. This is a key point of the proof, because at these mesh points a sharp almost sure upper bound on the growth rate of the process is known. In this paper however, due to the uncertainty of the differentiability of the resolvent of the SNFDE, we cannot apply the same analysis as in \cite{AppleWu08linearsfde}. However, part of the strategy of our proof involves writing the solution of the affine SNFDE in terms of the solution of an affine SFDE which does have an underlying deterministic differential resolvent which is continuously differentiable, enabling some of the techniques and estimates of \cite{AppleWu08linearsfde} to be employed once more.

More precisely, we study the asymptotic behaviour of the
finite--dimensional process which satisfies
\begin{subequations} \label{eq:linearneutral}
\begin{align} \label{eq.linearneutral1}
X(t)-D(X_t)&=\phi(0)-D(\phi_0)+\int_0^t L(X_s)\,ds + \int_0^t \Sigma\,dB(s),\quad
t\geq 0, \\
 \label{eq.linearneutral2}
X(t)&=\phi(t), \quad t\in[-\tau,0],
\end{align}
\end{subequations}
where $B$ is an $m$--dimensional standard Brownian motion, $\Sigma$
is a $d\times m$--matrix with real entries, and
$D,L:C[-\tau,0]\to\mathbb{R}^d$ are linear functional with $\tau\geq
0$ and
\[
L(\phi)=\int_{[-\tau,0]} \nu(ds) \phi(s), \quad
D(\phi)=\int_{[-\tau,0]} \mu(ds) \phi(s), \quad
\phi\in
C([-\tau,0];\mathbb{R}^d).
\]
The asymptotic behaviour of \eqref{eq:linearneutral} is determined in
the case when the resolvent $\rho$ of the deterministic equation
\begin{equation} \label{eq.introdet}
\frac{d}{dt}\left( x(t)-D(x_t)\right)=L(x_t), \quad t\geq 0
\end{equation}
obeys $\rho\in L^1([0,\infty);\mathbb{R}^{d\times d})$ and the integral resolvent of $-\mu_+\in M([0,\infty),\mathbb{R}^d)$ is a finite measure,
where $\mu_+(E)=\mu(-E)$ for every Borel subset $E$ of $[0,\infty)$, and $\mu(E)=0$ for all Borel sets $E\subset (-\infty,-\tau)$. In particular, we show that the running maxima of each component grows according to
\begin{equation}\label{eq.introcompn}
\limsup_{t\to\infty} \frac{\langle
X(t),\mathbf{e}_i\rangle}{\sqrt{2\log t}}=\sigma_i,\quad
\liminf_{t\to\infty} \frac{\langle
X(t),\mathbf{e}_i\rangle}{\sqrt{2\log t}}=-\sigma_i,\quad\text{a.s.}
\end{equation}
where $\sigma_i>0$ depends on $\Sigma$ and the resolvent $\rho$.
Moreover
\begin{equation}  \label{eq.introinftynormn}
\limsup_{t\to\infty} \frac{|X(t)|_\infty}{\sqrt{2\log
t}}=\max_{i=1,\ldots,d} \sigma_i,\quad\text{a.s.}
\end{equation}

We can also subject \eqref{eq:linearneutral} to a general nonlinear
perturbation to get the equation
\begin{equation} \label{eq.introstochnonn}
d(X(t)-N_1(t,X_t))=(L(X_t)+N_2(t,X_t))\,dt + \Sigma\,dB(t), \quad t\geq 0,
\end{equation}
and still retain the asymptotic behaviour of \eqref{eq:linearneutral}.
More specifically, if the nonlinear functionals $N_1,N_2:[0,\infty)\times
C[-\tau,0]\to\mathbb{R}^d$ is of smaller than linear order as
$\|\varphi\|_2:=\sup_{-\tau\leq s\leq 0}|\varphi(s)|_2\to\infty$ in
the sense that
\begin{equation} \label{eq.intrononlinfunctn}
\lim_{\|\varphi\|_2\to\infty}
\frac{|N_i(t,\phi_t)|_2}{\|\varphi\|_2}=0 \text{ uniformly in $t\geq
0$},
\end{equation}
then \eqref{eq.introcompn} and \eqref{eq.introinftynormn} still hold.

It should be remarked that we establish a stochastic variation of parameters formula for
solutions of \eqref{eq:linearneutral}. To the best of our knowledge, such a formula does not
appear in the literature to date. If $\rho$ is the differential resolvent of \eqref{eq.introdet}
and $x$ is the solution of \eqref{eq.introdet}, then the solution $X$ of \eqref{eq:linearneutral}
obeys
\begin{equation} \label{eq.introvarconst}
X(t)=x(t)+\int_0^t \rho(t-s)\Sigma\,dB(s), \quad t\geq 0,
\end{equation}
One interesting aspect of the proof of \eqref{eq.introvarconst} is that it can be applied to equations
with non--constant diffusion coefficient. We intend to give some applications of this result to characterise
the asymptotic behaviour of SNFDEs in later work.

Neutral delay differential equations have been used to describe
various processes in physics and engineering sciences
\cite{HaleLun93}, \cite{Stepan:1989}. For example, transmission
lines involving nonlinear boundary conditions~\cite{Hale:77,WuXia96}
cell growth dynamics \cite{BakBochPaul:1998}, propagating pulses in
cardiac tissue \cite{CourGlassKeen:1993} and drill--string vibrations
\cite{BalJansMcClint:2003} have been described by means of neutral
delay differential equations. Reliable simulation of such equations in applications
in which stochastic perturbations are present is facilitated by Euler--Maruyama methods for
SNFDEs developed by Mao and Wu~\cite{MaoWu:2008}.


\section{Preliminaries}\label{sec:prelimiaries}
Let $d,m$ be some positive integers and $\mathbb{R}^{d\times m}$
denote the space of all $d \times m$ matrices with real entries. We
equip $\mathbb{R}^{d\times m}$ with a norm $\abs{\cdot}$ and write
$\mathbb{R}^d$ if $m=1$ and $\mathbb{R}$ if $d=m=1$. We denote by
$\mathbb{R}^+$ the half-line $[0,\infty)$. The complex plane is
denoted by $\mathbb{C}$.

The total variation of a measure $\nu$ in
$M([-\tau,0],\mathbb{R}^{d\times d})$ on a Borel set $B\subseteq
[-\tau,0]$ is defined by
\begin{align*}
 \abs{\nu}\!(B):=\sup\sum_{i=1}^N \abs{\nu(E_i)},
\end{align*}
where $(E_i)_{i=1}^N$ is a partition of $B$ and the supremum is
taken over all partitions. The total variation defines a positive
scalar measure $\abs{\nu}$ in $M([-\tau,0],\mathbb{R})$. If one
specifies temporarily the norm $\abs{\cdot}$ as the $l^1$-norm on
the space of real-valued sequences and identifies
$\mathbb{R}^{d\times d}$ by $\mathbb{R}^{d^2}$ one can easily
establish for the measure $\nu=(\nu_{i,j})_{i,j=1}^{d^2}$ the
inequality
\begin{align}\label{eq.totalvarest}
 \abs{\nu}\!(B)\le  C \sum_{i=1}^d\sum_{j=1}^d \abs{\nu_{i,j}}\!(B)
 \qquad\text{for every Borel set }B\subseteq [-\tau,0]
\end{align}
with $C=1$. Then, by the equivalence of every norm on
finite-dimensional spaces, the inequality \eqref{eq.totalvarest}
holds true for the arbitrary norms $\abs{\cdot}$ and some constant
$C>0$. Moreover, as in the scalar case we have the fundamental
estimate
\begin{align*}
\abs{\int_{[-\tau,0]} \nu(ds)\, f(s)} \le
\int_{[-\tau,0]}\abs{f(s)}\,\abs{\nu}\!(ds)
\end{align*}
for every function $f:[-\tau,0]\to\mathbb{R}^{d \times d^{\prime}}$
which is $\abs{\nu}$-integrable.

We first turn our attention to the deterministic delay equation
underlying the stochastic differential equation
\eqref{eq:linearneutral}. For a fixed constant $\tau\geq 0$ we
consider the deterministic linear delay differential equation
\begin{align}
  \begin{split}
     \frac{d}{dt}\bigg(x(t)-\int_{[-\tau,0]}\mu(ds)x(t+s)\bigg) &= \int_{[-\tau,0]} \nu(ds)\,x(t+s),   \quad\text{for }t\geq 0,\\
     x(t) &= \phi(t) \quad \text{for }t\in [-\tau,0],
  \end{split}\label{eq:deterministic}
\end{align}
for measures $\nu\in M([-\tau,0],\mathbb{R}^{d\times d})$, $\mu\in
M([-\tau,0],\mathbb{R}^{d\times d})$. The initial function $\phi$ is
assumed to be in the space
$C[-\tau,0]:=\{\phi:[-\tau,0]\to\mathbb{R}^d: \text{continuous}\}$.
A function $x:[-\tau,\infty)\to\mathbb{R}^d$ is called a {\em
solution} of \eqref{eq:deterministic} if $x$ is continuous on
$[-\tau,\infty)$ and $x$ satisfies the first and second identity of
\eqref{eq:deterministic} for all $t\geq 0$ and $t\in [-\tau,0]$,
respectively. It is well--known that for every $\phi\in C[-\tau,0]$
the problem \eqref{eq:deterministic} admits a unique solution $x=
x(\cdot,\phi)$ provided that $\det(I_d-\mu(\{0\}))\neq 0$, where
$I_d$ is the $d\times d$ identity matrix, and $\det(A)$ signifies
the determinant of a $d\times d$ matrix $A$. This condition on $\mu$ is equivalent to the notion of
\emph{uniform non--atomicity at $0$} of the functional $D:C[-\tau,0]\to \mathbb{R}^d$ given by
\[
D(\psi) = \int_{[-\tau,0]} \mu(ds) \psi(s), \quad \psi\in C([-\tau,0];\mathbb{R}^d).
\]
Results on the existence of deterministic neutral equations, including a definition of uniform non--atomicity of $D$, may be found in Chukwu~\cite{Chukwu:1994}, Chukwu and Simpsion~\cite{ChukwuSimp:1989}, Hale~\cite{Hale:71} and Hale and Cruz~\cite{HaleCruz:70}.

The {\em fundamental solution} or {\em resolvent} of
\eqref{eq:deterministic} is the unique continuous function $\rho:
[0,\infty)\to\mathbb{R}$ which satisfies
\begin{gather}\label{eq:fundaneu}
\frac{d}{dt}\left(\rho(t)-\int_{[-\tau,0]}\mu(ds)\rho(t+s)\right)=
\left(\int_{[-\tau,0]}\nu(ds)\rho(t+s)\right),\quad t\geq 0;\\
\nonumber \rho(t)=0,\quad t\in[-\tau,0);\quad \rho(0)=I_d.
\end{gather}
It plays a role which is analogous to the fundamental system in
linear ordinary differential equations and the Green function in
partial differential equations.

For a function $x:[-\tau,\infty)\to \mathbb{R}^d$ we denote the
\emph{segment of $x$} at time $t\geq 0$ by the function
\begin{align*}
x_t:[-\tau,0]\to \mathbb{R}, \qquad x_t(s):=x(t+s).
\end{align*}
If we equip the space $C[-\tau,0]$ of continuous functions with the
supremum norm Riesz' representation theorem guarantees that every
continuous functional $D:C[-\tau,0]\to\mathbb{R}^{d\times d}$ is of
the form
\begin{align*}
 D(\psi)=\int_{[-\tau,0]}\mu(ds)\,\psi(s),
\end{align*}
for a measure $\mu\in M([-\tau,0];\mathbb{R}^d)$. Hence, we will
write \eqref{eq:deterministic} in the form
\begin{align*}
\frac{d}{dt}[x(t)-D(x_t)]=L(x_t)\quad \text{for }t\geq 0,\qquad x_0
&= \phi
\end{align*}
where
\[
L(\psi)=\int_{[-\tau,0]}\nu(ds)\,\psi(s),
\]
$\nu$ is a measure in $\mu\in M([-\tau,0];\mathbb{R}^d)$, and assume
$D$ and $L$ to be continuous and linear functionals on
$C([-\tau,0];\mathbb{R}^d)$.

Fix a complete probability space $(\Omega,\mathcal{F},\mathbb{P})$
with a filtration $\{\mathcal{F}(t)\}_{t\geq 0}$ satisfying the
usual conditions and let $(B(t):\,t\geq 0)$ be a standard
$m$--dimensional Brownian motion on this space. Equation
\eqref{eq:linearneutral} can be written as
\begin{align}
\begin{split}
 d[X(t)-D(X_t)]&= L(X_t)\,dt+ \Sigma\,dB(t)\quad\text{for }t\geq 0,\\
  X(t)&= \phi(t)\quad\text{for } t\in [-\tau,0],
 \end{split}\label{eq:snde}
\end{align}
where $D$ and $L$ are as previously defined, and
$\Sigma\in\mathbb{R}^{d\times m}$.

In~\cite{ApplebyMaoWu09neuexi}, we discussed the conditions under which
\eqref{eq:snde} has a unique solution on $[-\tau, T]$ for any $T>0$.
In addition to the Lipschitz continuity of the linear functional
$L$, we require that the neutral functional $D$ is \emph{uniformly
nonatomic} at zero. In the scalar case, it can be shown that if
$\mu(\{0\})=1$, $D$ does not obey the uniformly nonatomic condition.
For $\mu(\{0\})\in \mathbb{R}/\{1\}$, \eqref{eq:snde} can be
rescaled, so that a unique solution exists. In the general
finite--dimensional case, if $I_d-\mu(\{0\})$ is invertible, we can
rescale the equation so that the neutral functional $\tilde{D}$ is
given by
\[
\tilde{D}(\phi)=(I_d-\mu(\{0\}))^{-1}\int_{[-\tau,0)}
\mu(ds)\phi(s)=:\int_{[-\tau,0]} \tilde{\mu}(ds)\phi(s), \quad
\phi\in C([-\tau,0];\mathbb{R}^d),
\]
where $\tilde{\mu}\in M$. Then $\tilde{D}$ is uniformly non--atomic
at zero, and indeed $\tilde{\mu}(\{0\})=0$. Hence without loss of
generality, we can assume that
\begin{equation}  \label{eq.mu0zero}
\mu(\{0\})=0.
\end{equation}

The dependence of the solutions on the initial condition $\phi$ is
neglected in our notation in what follows; that is, we will write
$x(t)=x(t,\phi)$ and $X(t)=X(t,\phi)$ for the solutions of
\eqref{eq:deterministic} and \eqref{eq:snde} respectively.

We also constrain ourselves with the condition
\begin{equation}\label{integresolventcondition}
\det\left(I_d-\int_{[-\tau,0]}e^{\lambda s}\mu(ds)\right)\neq 0\quad
\text{for every }\, \lambda\in\mathbb{C}\, \text{ with }\,\Re
\lambda\geq 0.
\end{equation}
Define the function $h_{\mu,\nu}:\mathbb{C}\to\mathbb{C}$ by
\[
h_{\mu,\nu}(\lambda)=\det\left(\lambda\bigg(I_d-\int_{[-\tau,0]}e^{\lambda
s}\mu(ds)\bigg)-\int_{[-\tau,0]}e^{\lambda s}\nu(ds)\right).
\]
The asymptotic behaviour of $\rho$ relies on the value of
\begin{equation}\label{suprepart}
v_0(\mu, \nu):=\sup\bigg\{\Re(\lambda):
\lambda\in\mathbb{C},h_{\mu,\nu}(\lambda)=0\bigg\}.
\end{equation}
We summarize some conditions on the asymptotic behaviour of $\rho$
in the following lemma:
\begin{lemma}\label{lemmaneutralequ}
Let $\rho$ satisfy \eqref{eq:fundaneu}, and $v_0(\mu, \nu)$ be
defined as \eqref{suprepart}. If \eqref{integresolventcondition}
holds, then the following statements are equivalent:
\begin{itemize}
 \item [(a)] $v_0(\mu, \nu)<0$.
 \item [(b)] $\rho$ decays to zero exponentially.
 \item [(c)] $\rho(t)\to 0$ as $t\to \infty$.
 \item [(d)] $\rho\in L^1(\mathbb{R}^+;\mathbb{R}^{d\times d})$.
 \item [(e)] $\rho\in L^2(\mathbb{R}^+;\mathbb{R}^{d\times d})$.
\end{itemize}
\end{lemma}

Since \eqref{eq:snde} may be viewed as a perturbation of the
non--stochastic equation \eqref{eq:deterministic}, it is natural to
expect that its solution can be written in terms of $x$ and the
differential resolvent $\rho$ via a variation of constants formula.
We show below that the solution of \eqref{eq:linearneutral} has the
following representation
\begin{theorem} \label{thm.varconstneut}
Suppose that $L$ and $D$ are linear functionals and that $\mu$ obeys
\eqref{eq.mu0zero}. If $x$ is the solution of
\eqref{eq:deterministic} and $\rho$ is the continuous solution of
\eqref{eq:fundaneu}, then the unique continuous adapted process $X$
which satisfies \eqref{eq:snde} obeys
\begin{equation}\label{eq:neusolution}
X(t)=x(t)+\int_0^t\rho(t-s)\Sigma\,dB(s),\quad t\geq 0,
\end{equation}
and $X(t)=\phi(t)$ for $t\in[-\tau,0]$.
\end{theorem}

In the paper, we let $\langle \cdot,\cdot\rangle$ stand for the
standard inner product on $\mathbb{R}^d$, and $|\cdot|_2$ for the
standard Euclidean norm induced from it. We also let
$|\cdot|_\infty$ stand for the infinity norm on $\mathbb{R}^d$, and
if $\phi\in C([-\tau,0];\mathbb{R}^d)$, we define
$\|\phi\|_2=\sup_{-\tau\leq s\leq 0}|\phi(s)|_2$. By way of
clarification, we note that $|\cdot|_\infty$ stands here for a
vector norm rather than a norm on a space of continuous functions.
For $i=1,\ldots, d$, the $i$-th standard basis vector in
$\mathbb{R}^d$ is denoted $\mathbf{e}_i$.
If $X$ and $Y$ are two random variables, then we
denote the correlation and the covariance between $X$ and $Y$ by
$\text{Corr}(X,Y)$ and $\Cov(X,Y)$ respectively.

\section{Statement and Discussion of Main Results}\label{sec:mainresults}

In this section the main asymptotic results of the paper are stated.
We start by stating our main result for the solution of the
finite--dimensional affine equation \eqref{eq:snde}.
\begin{theorem}\label{finidimneut}
Suppose that $\rho$ is the solution of \eqref{eq:fundaneu} and that
$\mu$ satisfies \eqref{integresolventcondition}. Moreover suppose
that $v_0(\mu,\nu)<0$, where $v_0(\mu,\nu)$ is defined by
\eqref{suprepart}. Let $X$ be the unique continuous adapted
$d$-dimensional process which obeys \eqref{eq:snde}. Then for each
$1\leq i \leq d$,
\begin{equation}\label{eq:finidimlimsupcompneut}
\limsup_{t\to \infty}\frac{X_i(t)}{\sqrt{2\log t}}= \sigma_i\quad
\text{and}\quad \liminf_{t\to \infty}\frac{X_i(t)}{\sqrt{2\log t}}=
-\sigma_i,\quad \text{a.s.}
\end{equation}
where $\sigma_i$ is given by
\begin{equation} \label{eq.sigmaineut}
\sigma_i^2=\int_0^\infty \sum_{k=1}^m \theta_{ik}(s)^2\,ds
\end{equation}
where $\theta(t)=\rho(t)\Sigma\in \mathbb{R}^{d\times m}$. Moreover
\begin{equation}\label{eq:finidimlimsupneut}
\limsup_{t\to \infty}\frac{|X(t)|_\infty}{\sqrt{2\log t}}=
\max_{i=1,\ldots,d}\sigma_i,\quad \text{a.s.}
\end{equation}
\end{theorem}
The results of Theorem~\ref{finidimneut} are very similar to those
of Theorem 2 in~\cite{AppleWu08linearsfde} which considers the affine
functional differential equation
\begin{equation}\label{eq.stochSFDE}
dX(t)=L(X_t)\,dt + \Sigma\,dB(t), \quad t \geq 0; \quad X_0=\phi.
\end{equation}
We will refer to this result throughout the paper, so it is stated
shortly for convenience. To do so we need some auxiliary
deterministic functions. Define the differential resolvent $r$ by
\begin{equation} \label{eq:r'(t)}
r'(t)=L(r_t), \quad t\geq 0; \quad r(0)=I_d, \quad r(t)=0, \quad
t\in[-\tau,0).
\end{equation}
We define $h_\nu:\mathbb{C}\to\mathbb{C}$ by
\[
h_{\nu}(\lambda)=\det\left(\lambda I_d - \int_{[-\tau,0]} e^{\lambda
s}\nu(ds)\right),
\]
and suppose that
\begin{equation} \label{eq.SFDEv_0}
v_0(\nu):=\sup\{ \Re(\lambda): h_\nu(\lambda)=0\}<0.
\end{equation}
Theorem 2 of~\cite{AppleWu08linearsfde} is as follows.
\begin{theorem}\label{finidim}
Suppose that $r$ is the solution of \eqref{eq:r'(t)} and that
$v_0(\nu)<0$, where $v_0(\nu)$ is defined as \eqref{eq.SFDEv_0}. Let
$X$ be the unique continuous adapted $d$-dimensional process which
obeys \eqref{eq.stochSFDE}. Then for each $1\leq i \leq d$,
\begin{equation}\label{eq:finidimlimsupcomp}
\limsup_{t\to \infty}\frac{X_i(t)}{\sqrt{2\log t}}= \sigma_i\quad
\text{and}\quad \liminf_{t\to \infty}\frac{X_i(t)}{\sqrt{2\log t}}=
-\sigma_i,\quad \text{a.s.}
\end{equation}
where
\begin{equation} \label{eq.sigmaiSFDE}
\sigma_i= \sqrt{\sum^{m}_{k=1}\int ^\infty_0 \theta^2_{ik}(s)\,ds}
\end{equation}
and $\theta(t)=r(t)\Sigma\in \mathbb{R}^{d\times m}$. Moreover
\begin{equation}\label{eq:finidimlimsupSFDE}
\limsup_{t\to \infty}\frac{|X(t)|_\infty}{\sqrt{2\log t}}=
\max_{i=1,\ldots,d}\sigma_i,\quad \text{a.s.}
\end{equation}
\end{theorem}
In other words, the solution $X$ of \eqref{eq.stochSFDE} obeys all
the conclusions of Theorem~\ref{finidimneut} above with $r$ in place
of $\rho$.

The proof of Theorem~\ref{finidim} depends on two key properties of
the differential resolvent $r$ satisfying \eqref{eq:r'(t)}. The
first is that $r$ decays exponentially fast because $v_0(\nu)<0$.
This is in common with the condition $v_0(\mu, \nu)<0$ in
Theorem~\ref{finidimneut}. The second is that $r$ is in
$C^1((0,\infty);\mathbb{R}^d)$, which plays a crucial role in the
proof of Theorem~\ref{finidim} in controlling the behaviour of the
process between mesh points. In contrast with the differentiability
of $r$, the neutral differential resolvent $\rho$ may not be
differentiable everywhere on $(0,\infty)$. Therefore the proof of
Theorem \ref{finidimneut} departs from that of Theorem~\ref{finidim}
in controlling the behaviour of the process between mesh points.

Our other main result shows that \eqref{eq:snde} can be perturbed by
nonlinear functionals $N_1$ and $N_2$ in the neutral term and drift
respectively (which is of lower than linear order at infinity)
without changing the asymptotic behaviour of the underlying affine
stochastic neutral functional differential equation. To make this
claim more precise, we characterize the perturbing nonlinear
functionals $N_1$ and $N_2$ as follows: suppose
$\{N_i\}_{i=1,2}:[0,\infty)\times C[-\tau,0]\to\mathbb{R}^d$ obey
\begin{equation}
\begin{split}
\text{For all $n\in\mathbb{N}$ there exists a $K_n>0$ such that
if $\varphi$, $\psi\in C([-\tau,0];\mathbb{R}^d)$} \\
 \text{obey $\|\varphi\|_2\vee \|\psi\|_2\leq n$,
then } |N_i(t,\varphi)-N_i(t,\psi)|_2\leq K_n\|\varphi-\psi\|_2, \\
\text{and $N_i$ is continuous in its first argument for $i=1,2$;}
\end{split}\label{eq:n1}
\end{equation}
\begin{equation}\label{eq:n2} \lim_{\|\varphi\|_2\to\infty}
\frac{|N_i(t,\varphi)|_2}{\|\varphi\|_2}=0 \quad \text{uniformly in
$t$ for $i=1,2$},
\end{equation}
and
\begin{equation} \label{eq:n3}
t\mapsto |N_i(t,0)|_2 \quad\text{is bounded on $[0,\infty)$ for
$i=1,2$.}
\end{equation}
Before stating our main result, we examine the hypotheses
\eqref{eq:n1}--\eqref{eq:n3} and prove an important estimate
deriving therefrom. By the hypothesis \eqref{eq:n2} we mean that for
every $\varepsilon>0$ there is a $\Phi=\Phi(\varepsilon)>0$ such
that if $\varphi\in C([-\tau,0];\mathbb{R}^d)$ obeys
$\|\varphi\|_2\geq \Phi(\varepsilon)$, we then have
\[
|N_i(t,\varphi)|_2\leq \varepsilon \|\varphi\|_2, \quad \text{for
all $t\geq 0$ and $i=1,2$}.
\]
By \eqref{eq:n3}, we have that there is a $\bar{n}\geq 0$ such that
$|N_i(t,0)|_2\leq \bar{n}$ for all $t\geq 0$. Also by \eqref{eq:n1}
for all $\varphi$ such that $\|\varphi\|_2\leq
\lceil\Phi(\varepsilon)\rceil$ (where $\lceil x\rceil$ denotes the
smallest integer greater than or equal to $x\geq 0$), we have that
there is a $K(\varepsilon)=K_{\lceil\Phi(\varepsilon)\rceil}$ such
that
\[
|N_i(t,\varphi)|_2\leq |N_i(t,\varphi)-N_i(t,0)|_2+|N_i(t,0)|_2\leq
K(\varepsilon)\|\varphi\|_2 + \bar{n}\leq
K(\varepsilon)\lceil\Phi(\varepsilon)\rceil+\bar{n}.
\]
Therefore with
$L(\varepsilon):=K(\varepsilon)\lceil\Phi(\varepsilon)\rceil+\bar{n}$
we have
\[
|N_i(t,\varphi)|_2\leq L(\varepsilon), \quad\text{for all $t\geq 0$
and all $\|\varphi\|_2\leq \Phi(\varepsilon)$}.
\]
Hence for every $\varepsilon>0$ there exists $L(\varepsilon)>0$ such
that
\begin{equation} \label{eq.nestimate}
|N_i(t,\varphi)|_2\leq L(\varepsilon)+\varepsilon\|\varphi\|_2,
\quad \text{for all $t\geq 0$ and all $\varphi\in
C([-\tau,0];\mathbb{R}^d)$}.
\end{equation}
The hypothesis \eqref{eq:n3} ensures that the functional $N_i$ is
(in some sense) close to being an autonomous functional, or is
bounded by an autonomous functional.
%
%
%

We study the following nonlinear stochastic differential equation
with time delay:
\begin{align}
\begin{split}
 d[X(t)-D(X_t)-N_1(t,X_t)]&= [L(X_t)+N_2(t,X_t)]\,dt+ \Sigma\,dB(t)\quad\text{for }t\geq 0,\\
  X(t)&= \phi(t)\quad\text{for } t\in [-\tau,0],
 \end{split}\label{eq.nonlinearsfde}
\end{align}
where $D$ and $L$ are continuous and linear functionals on
$C([-\tau,0];\mathbb{R})$ as defined in the preliminaries.

The following theorem is a consequence of the affine
finite--dimensional result Theorem~\ref{finidimneut}.

\begin{theorem} \label{thm.nonlinrandwalkneut}
Suppose that $N_1$ and $N_2$ obey \eqref{eq:n1}, \eqref{eq:n2} and
\eqref{eq:n3} and that $N_1$ is uniformly nonatomic at $0$. Also
suppose that $\rho$ is the solution of \eqref{eq:fundaneu} and that
$\mu$ satisfies \eqref{integresolventcondition}. Moreover suppose
that $v_0(\mu,\nu)<0$, where $v_0(\mu,\nu)$ is defined by
\eqref{suprepart}. Let $X$ be the unique continuous adapted
$d$-dimensional process which obeys \eqref{eq.nonlinearsfde} Then
for each $1\leq i \leq d$,
\begin{equation}\label{eq:finidimlimsupcompnon}
\limsup_{t\to \infty}\frac{X_i(t)}{\sqrt{2\log t}}=\sigma_i,\quad
\text{and}\quad \liminf_{t\to \infty}\frac{X_i(t)}{\sqrt{2\log
t}}=-\sigma_i,\quad\text{a.s.}
\end{equation}
where $\sigma_i$ is given \eqref{eq.sigmaineut}. Moreover
\begin{equation}\label{eq:finidimlimsupnon}
\limsup_{t\to \infty}\frac{|X(t)|_\infty}{\sqrt{2\log t}}=
\max_{1\leq i\leq d} \sigma_i,
\quad \text{a.s.}
\end{equation}
\end{theorem}
Since in general, it is not possible to obtain a representation that
is analogous to \eqref{eq:neusolution} for non-linear equations such
as \eqref{eq.nonlinearsfde}, the proof cannot directly rely on
Gaussianity of the process. Instead, by using a \emph{comparison}
argument, we conclude that if the non-linear term in the drift is
smaller than linear order at infinity (cf. assumption
\eqref{eq:n2}), the size of the large fluctuations of a Gaussian
stationary process is retained.

\section{Proofs from Section \ref{sec:mainresults}}\label{sec:prooftheorems}
\subsection{Proof of Lemma~\ref{lemmaneutralequ}}
We extend the measures $\mu$ and $\nu$ to
$M((-\infty,0];\mathbb{R}^{d\times d})$ by assuming
\begin{equation} \label{eq.extendmunu}
\mu(E)=\nu(E)=0\quad\text{for every Borel set}\,\, E\subseteq
(-\infty,-\tau).
\end{equation}
For any Borel set $E\subseteq \mathbb{R}$ we use the notation
\[
-E:=\{x\in\mathbb{R}:-x\in E\}
\]
to define the reflected Borel set $(-E)$. Introduce the measures
$\mu_+$ and $\nu_+$ in $M([0,\infty);\mathbb{R}^{d\times d})$,
related to $\mu$ and $\nu$ in $M((-\infty,0];\mathbb{R}^{d\times
d})$ by
\begin{equation}\label{def.muplusnuplus}
\mu_+(E):=\mu(-E),\quad \nu_+(E):=\nu(-E).
\end{equation}
Then for $t\geq 0$,
\begin{align}\label{muplus}
\int_{[-\tau,0]}\mu(ds)\rho(t+s)&=\int_{[0,\tau]}\mu_+(ds)\rho(t-s)\\
\nonumber
&=\int_{[0,\infty)}\mu_+(ds)\rho(t-s)-\int_{(\tau,\infty)}\mu_+(ds)\rho(t-s)\\
\nonumber &=\int_{[0,\infty)}\mu_+(ds)\rho(t-s)\\ \nonumber
&=\int_{[0,t]}\mu_+(ds)\rho(t-s)+\int_{(t,\infty)}\mu_+(ds)\rho(t-s)\\
\nonumber &=\int_{[0,t]}\mu_+(ds)\rho(t-s).
\end{align}
The last step is obtained by the fact that $\rho(t)=0$ for $t\in
[-\tau,0)$ and $\mu(\{0\})=0$. Similarly
\begin{equation}\label{etaplus}
\int_{[-\tau,0]}\nu(ds)\rho(t+s)=\int_{[0,t]}\nu_+(ds)\rho(t-s),\quad
t\geq 0.
\end{equation}
Define
\begin{displaymath}\label{def.kappa}
\kappa(t):=\left\{\begin{array}{ll}\rho(t)-\int_{[-\tau,0]}\mu(ds)\rho(t+s),& t\geq 0,\\
0,&t\in [-\tau,0).\end{array}\right.
\end{displaymath}
Also since $\rho(0)=I_d$ and $\mu(\{0\})=0$, $\kappa(0)=I_d$.
Moreover, by \eqref{muplus} and \eqref{etaplus}, we have
\begin{equation}\label{eq.kappainrho}
\kappa=\rho-\mu_+\ast\rho,
\end{equation}
and $\kappa\in C^1((0,\infty);\mathbb{R}^{d\times d})$ with
\begin{equation}\label{kappaderi}
\kappa'(t)=\int_{[0,t]}\nu_+(ds)\rho(t-s),\quad t\geq 0.
\end{equation}

Since $\mu_+\in M_([0,\infty);\mathbb{R}^{d\times d})$ and
$\mu_+(\{0\})=\mu(\{0\})=0$, we may define $\rho_0\in
M_{\text{loc}}([0,\infty);\mathbb{R}^{d\times d})$ to be the
integral resolvent of $(-\mu_+)$, i.e.,
\begin{equation}\label{integralre}
\rho_0-\mu_+\ast \rho_0=-\mu_+.
\end{equation}
Then by Theorem 4.1.7 in \cite{GripenbergLonden90},
\begin{equation}\label{rhoinkappa}
\rho=\kappa-\rho_0\ast \kappa.
\end{equation}
Therefore, if we define $\beta\in
M_{\text{loc}}([0,\infty);\mathbb{R}^{d\times d}$ by
\begin{equation} \label{def.beta}
\beta:=\nu_+-\nu_+\ast\rho_0,
\end{equation}
we have that
\begin{equation} \label{eq.kappadiffresolv}
\kappa'(t)=(\beta\ast\kappa)(t), \quad t\geq 0; \quad \kappa(0)=I_d.
\end{equation}

By condition \eqref{integresolventcondition} and by Theorem 4.1.5
(half line Paley--Wiener theorem) in \cite{GripenbergLonden90}, we
have that $\rho_0$ defined by \eqref{integralre} obeys $\rho_0\in
M([0,\infty);\mathbb{R}^{d\times d})$. Moreover, it is even true
that condition \eqref{integresolventcondition} implies that $\rho_0$
decays exponentially, so
\begin{equation}\label{rho0decaysexp}\text{there exists $\alpha>0$ such
that } \int_{[0,\infty)}e^{\alpha t}|\rho_0|(dt)<\infty.
\end{equation}
Therefore by using Theorem 3.6.1 in \cite{GripenbergLonden90},
\eqref{eq.kappainrho}, \eqref{rhoinkappa}, and \eqref{rho0decaysexp}
we have the following equivalences
\begin{align}\label{kapparhoiff1}
\lim_{t\to\infty}\kappa(t)=0\quad&\Leftrightarrow \quad\lim_{t\to\infty}\rho(t)=0;\\
\label{kapparhoiff2}
\kappa \,\,\text{decays to zero exponentially}\quad&\Leftrightarrow \quad \rho \,\,\text{decays to zero exponentially;}\\
\label{kapparhoiff3}
\kappa\in L^1(\mathbb{R}^+;\mathbb{R}^{d\times d})\quad&\Leftrightarrow \quad\rho\in L^1(\mathbb{R}^+;\mathbb{R}^{d\times d});\\
\label{kapparhoiff4} \kappa\in L^2(\mathbb{R}^+;\mathbb{R}^{d\times
d})\quad&\Leftrightarrow \quad\rho\in
L^2(\mathbb{R}^+;\mathbb{R}^{d\times d}).
\end{align}
Since $\rho_0\in M([0,\infty);\mathbb{R}^{d\times d})$, we have that
$\beta$ defined by \eqref{def.beta} is in
$M(\mathbb{R}^+;\mathbb{R}^{d\times d})$. Now by Theorem 3.3.17 from
\cite{GripenbergLonden90}, if $\beta$ has \emph{a fortiori} a finite
first moment, i.e. $\int_{[0,\infty)}t|\beta|(dt)<\infty$, where
$\beta:=\nu_+-\nu_+\ast\rho_0$, then
\begin{equation}\label{part1iff}
\lim_{t\to\infty}\kappa(t)=0\quad \Leftrightarrow \quad \kappa\in
L^1(\mathbb{R}^+;\mathbb{R}^{d\times d}).
\end{equation}
We now show that $\beta$ has a finite first moment. Note that
\begin{align*}
\int_{[0,\infty)}t|\beta|(dt)&\leq \int_{[0,\infty)}t|\nu_+|(dt)+\int_{[0,\infty)}t|\nu_+\ast\rho_0|(dt)\\
&=\int_{[0,\tau]}t|\nu_+|(dt)+\int_{[0,\infty)}t|\nu_+\ast\rho_0|(dt).
\end{align*}
 Thus by Young's inequality,
\begin{align*}
\int_{[0,\infty)}t|\nu_+\ast\rho_0|(dt)&\leq \frac{1}{\alpha}\int_{[0,\infty)}e^{\alpha t}|\nu_+\ast\rho_0|(dt)\\
&\leq \frac{1}{\alpha}\int_{[0,\infty)}e^{\alpha t}|\nu_+|(dt)\int_{[0,\infty)}e^{\alpha t}|\rho_0|(dt)\\
&=\frac{1}{\alpha}\int_{[0,\tau]}e^{\alpha t}|\nu_+|(dt)\int_{[0,\infty)}e^{\alpha t}|\rho_0|(dt)\\
&< \infty.
\end{align*}
So $\beta$ has finite first moment. Therefore \eqref{part1iff}
holds. Moreover,
\begin{equation}\label{expbeta}
\int_{[0,\infty)}e^{\alpha t}|\beta|(dt)<\infty.
\end{equation}
So by \eqref{kapparhoiff1}, \eqref{kapparhoiff3} and
\eqref{part1iff}, statements (c) and (d) are equivalent. Now if
$\kappa\in L^1(\mathbb{R}^+;\mathbb{R}^{d\times d})$, due to
\eqref{expbeta}, we have that $\kappa$ decays to zero exponentially,
which by \eqref{kapparhoiff2} implies that $\rho$ decays to zero
exponentially. Hence (b) and (c) are equivalent. If
$\lim_{t\to\infty}\rho(t)=0$, then $\rho$ decays to zero
exponentially, which implies $\rho\in
L^2(\mathbb{R}^+;\mathbb{R}^{d\times d})$. On the other hand, if
$\rho\in L^2(\mathbb{R}^+;\mathbb{R}^{d\times d})$, then $\kappa\in
L^2(\mathbb{R}^+;\mathbb{R}^{d\times d})$. Also $\kappa'\in
L^2(\mathbb{R}^+;\mathbb{R}^{d\times d})$. Let
$f:=|\kappa|_F^2=\sum_{i=1}^d \sum_{j=1}^d \kappa_{ij}^2$. Then
\begin{align*}
|f'|&=\left|\sum_{i=1}^d \sum_{j=1}^d
2\kappa_{ij}(t)\kappa_{ij}'(t)\right| \leq \sum_{i=1}^d \sum_{j=1}^d
2|\kappa_{ij}(t)\kappa_{ij}'(t)|\\
&\leq \sum_{i=1}^d \sum_{j=1}^d \left(|\kappa_{ij}(t)|^2 +
|\kappa_{ij}'(t)|^2\right). \end{align*} Therefore $|f'|\leq
|\kappa|^2_F+|\kappa'|^2_F$, so $f'\in
L^1(\mathbb{R}^+;\mathbb{R})$. Therefore as $f\in
L^1(\mathbb{R}^+;\mathbb{R})$, we have
$\lim_{t\to\infty}\kappa(t)=0$ and consequently
$\lim_{t\to\infty}\rho(t)=0$. Hence (b)--(d) are equivalent. For
part (a), suppose $\rho\in L^1([0,\infty);\mathbb{R}^{d\times d})$,
which holds if and only if $\kappa\in
L^1([0,\infty);\mathbb{R}^{d\times d})$, which in turn is equivalent
to
\begin{equation}\label{betalarp}
\det(\lambda I_d-\hat{\beta}(\lambda))\neq 0,\quad \Re(\lambda)\geq
0.
\end{equation}
Now
\[
\hat{\rho}_0(\lambda)=-\left(I_d-\hat{\mu}_+(\lambda)\right)^{-1}\hat{\mu}_+(\lambda)
=-\hat{\mu}_+(\lambda)\left(I_d-\hat{\mu}_+(\lambda)\right)^{-1}
\]
for all $\Re{\lambda}\geq 0$, because
$\det(I_d-\hat{\mu}_+(\lambda))\neq 0$ for all $\Re{\lambda}\geq 0$
due to \eqref{integresolventcondition}. We have, for $\Re
\lambda\geq 0$,
\begin{align*}
\lambda I_d-\hat{\beta}(\lambda)
&=\lambda I_d-\hat{\nu}_+(\lambda)+\hat{\nu}_+(\lambda)\hat{\rho}_0(\lambda)\\
&=\lambda I_d -\hat{\nu}_+(\lambda)-\hat{\nu}_+(\lambda)\hat{\mu}_+(\lambda)\left(I_d-\hat{\mu}_+(\lambda)\right)^{-1}\\
&= \bigg[\lambda(I_d-\hat{\mu}_+(\lambda))
-\hat{\nu}_+(\lambda)(I_d-\hat{\mu}_+(\lambda))-\hat{\nu}_+(\lambda)\hat{\mu}_+(\lambda)\bigg]\left(I_d-\hat{\mu}_+(\lambda)\right)^{-1}\\
&=\bigg[\lambda(I_d-\hat{\mu}_+(\lambda))
-\hat{\nu}_+(\lambda)\bigg]\left(I_d-\hat{\mu}_+(\lambda)\right)^{-1}\\
&=\bigg[\lambda\left(I_d-\int_{[-\tau,0]}e^{\lambda
s}\mu(ds)\right)-\int_{[-\tau,0]}e^{\lambda
s}\nu(ds)\bigg]\left(I_d-\hat{\mu}_+(\lambda)\right)^{-1}
\end{align*}
Clearly, under \eqref{integresolventcondition}, \eqref{betalarp}
holds if and only if
\[
\det\left( \lambda\left(I_d-\int_{[-\tau,0]}e^{\lambda
s}\mu(ds)\right)-\int_{[-\tau,0]}e^{\lambda s}\nu(ds)\right)\neq 0,
\quad \text{for all }\Re\lambda\geq 0
\]
which is true if and only if $v_0(\mu, \nu)<0$. Hence statements
(a)--(e) are all equivalent.
\subsection{Proof of Theorem~\ref{thm.varconstneut}}
We first show that the solution of \eqref{eq:snde} can be represented
in the form of \eqref{eq:neusolution}. 
Let $\mu_+$ and $\nu_+$ be as in \eqref{def.muplusnuplus}. Then for
$t\geq 0$, the solution $X$ of \eqref{eq:snde} satisfies
\begin{equation*}
d\left(X(t)-\int_{[0,t]}\mu_+(ds)X(t-s)\right)=\left(\int_{[0,t]}\nu_+(ds)X(t-s)\right)\,dt+\Sigma\,dB(t),
\end{equation*}
with $X(t)=\phi(t)$ for $t\in [-\tau,0]$. Let $x$ be the solution of
\eqref{eq:deterministic} with $x(t)=\phi(t)$ for $t\in [-\tau,0]$.
By \eqref{muplus} and \eqref{etaplus} we have that the fundamental
solution $\rho$ of \eqref{eq:fundaneu} satisfies
\begin{equation}\label{diffrho}
\frac{d}{dt}\left(\rho(t)-\int_{[0,t]}\mu_+(ds)\rho(t-s)\right)=\int_{[0,t]}\nu_+(ds)\rho(t-s),
\quad t\geq 0,
\end{equation}
with $\rho(t)=0$ for $t\in [-\tau,0)$ and $\rho(0)=I_d$. Define
$W(t):=X(t)-x(t)$ for $t\geq -\tau$. Then $W$ obeys
\begin{align}
\begin{split} \label{eq.sdeW}
d\bigg(W(t)-\int_{[0,t]}\mu_+(ds)W(t-s)\bigg)&=\int_{[0,t]}\nu_+(ds)W(t-s)\,dt+\Sigma\,dB(t),\quad t\geq 0;\\
W(t)&=0,\quad t\in[-\tau,0],
\end{split}
\end{align}
and is the unique solution of the above equation. With $\kappa$
defined by \eqref{def.kappa} we have \eqref{eq.kappainrho} with
$\kappa(0)=I_d$ and $\kappa(t)=0$ for all $t<0$. Let
\begin{equation}\label{def.Z}
Z(t):=W(t)-\int_{[0,t]}\mu_+(ds)W(t-s),\quad t\in\mathbb{R}.
\end{equation} Then $Z(0)=W(0)=0$, and we may write $Z=W-\mu_+\ast W$.
Clearly $Z$ is continuous. Let $\rho_0$ be the integral resolvent of
$-\mu_+$ defined by \eqref{integralre}. Then by Theorem 4.1.7 in
\cite{GripenbergLonden90} we have $W=Z-\rho_0\ast Z$. Therefore by
this, the definition of $\beta$ from \eqref{def.beta},
\eqref{eq.sdeW}, and \eqref{def.Z} we get
\begin{align*}
dZ(t)&=(\nu_+\ast W)(t)\,dt+\Sigma\,dB(t)=\big[\nu_+\ast(Z-\rho_0\ast Z)\big](t)\,dt+\Sigma\,dB(t)\\
&=(\beta\ast Z)(t)\,dt+\Sigma\,dB(t)
\end{align*}
Now by \eqref{eq.kappadiffresolv}, and using the fact that $Z(0)=0$,
we have by the variation of constants formula (cf.~\cite{ReRiGa06}) that $Z$ obeys
\begin{equation} \label{eq.varconstZkappa}
Z(t)=\int_0^t\kappa(t-s)\Sigma\,dB(s),\quad t\geq 0.
\end{equation}
Therefore as $W=Z-\rho_0\ast Z$ we have
\[
W(t)=\int_0^t\kappa(t-s)\Sigma\,dB(s)-
\int_{[0,t]}\rho_0(ds)\int_0^{t-s}\kappa(t-s-u)\Sigma\,dB(u).
\]
Hence by a stochastic Fubini theorem (cf., e.g., \cite[Ch. IV.6, Theorem 64]{protterbook})
 and \eqref{rhoinkappa}, we have
for all $t\geq 0$,
\begin{align*}
W(t)&=\int_0^t\kappa(t-s)\Sigma\,dB(s)-\int_{0}^t\int_{[0,t-u]}\rho_0(ds)\kappa(t-s-u)\Sigma\,dB(u)\\
&=\int_0^t\kappa(t-s)\Sigma\,dB(s)-\int_0^t(\rho_0\ast
\kappa)(t-u)\Sigma\,dB(u)=\int_0^t\rho(t-s)\Sigma\,dB(s).
\end{align*}
Since $X(t)=x(t)+W(t)$ we have \eqref{eq:neusolution} as required.

\section{Proof of Theorem~\ref{finidimneut}}\label{subsec:proofneutral}
To prove this result, we need a result about the rate of growth of the running maxima
of a sequence of standard normal random variables which have an exponentially decaying 
autocovariance function. It is Lemma 3 in~\cite{AppleWu08linearsfde}.
\begin{lemma}\label{theorem:discrete lower bound}
Suppose $(X_{n})_{n=1}^\infty$ is a sequence of jointly normal
standard random variables satisfying
\[
|\Cov(X_{i},X_{j})| \leq \lambda^{|i-j|}
\]
for some $\lambda \in (0,1)$. Then
\begin{equation}   \label{eq:discrete lower bound}
\lim_{n \rightarrow \infty} \frac{\max_{1 \leq j \leq n} X_{j}}
{\sqrt{2 \log n}} =1, \quad \text{a.s.}
\end{equation}
\end{lemma}
This result will be required in the proof of the following lemma. 
\begin{lemma} \label{lemma.gaussproc}
Let $B$ be an $m$--dimensional standard Brownian motion. Suppose
that for each $j=1,\ldots,m$, $\gamma_j$ is a deterministic function
such that $\gamma_j\in C([0,\infty);\mathbb{R})\cap
L^2([0,\infty);\mathbb{R}^{d\times d}$. Define
\[
U(t)=\sum_{j=1}^m \int_0^t \gamma_j(t-s)\,dB_j(s), \quad t\geq 0.
\]
Then
\begin{itemize}
\item[(a)] For every $\theta\in(0,1)$, there is an a.s. event $\Omega_\theta$ such that
\[
\limsup_{n\to\infty} \frac{|U(n^\theta)|}{\sqrt{2\log n^\theta}}\leq
\left(\frac{1}{\theta} \sum_{j=1}^m \int_0^\infty
\gamma_j^2(s)\,ds\right)^{1/2}, \quad\text{a.s. on $\Omega_\theta$}.
\]
\item[(b)] If there exists $c>0$ and $\alpha>0$ such that $|\gamma_j(t)|\leq ce^{-\alpha t}$ for all $t\geq 0$ and $j=1,\ldots,m$ then
\[
\limsup_{t\to\infty} \frac{|U(t)|}{\sqrt{2\log t}}\geq
\left(\sum_{j=1}^m \int_0^\infty \gamma_j^2(s)\,ds\right)^{1/2},
\quad\text{a.s.}
\]
Furthermore we have 
\begin{align}
\label{eq.UgtGlogt}
\limsup_{t\to\infty} \frac{U(t)}{\sqrt{2\log t}}&\geq
\left(\sum_{j=1}^m \int_0^\infty \gamma_j^2(s)\,ds\right)^{1/2},\quad \text{a.s.}\\
\label{eq.UltminGlogt}
\liminf_{t\to\infty} \frac{U(t)}{\sqrt{2\log t}}&\leq
-\left(\sum_{j=1}^m \int_0^\infty \gamma_j^2(s)\,ds\right)^{1/2},
\quad\text{a.s.}
\end{align}
\end{itemize}
\end{lemma}
\begin{proof}
Note that with $v$ defined by
\begin{equation}\label{def.v}
v^2(t):=\int_0^t \sum_{j=1}^m \gamma_j^2(s)\,ds, \quad\text{for all $t\geq 0$}.
\end{equation}
we have that $U(t)$ is normally distributed with mean zero and variance $v(t)$. 

We first prove part (a). In the case when $\sum_{j=1}^m \int_0^\infty \gamma_j^2(s)\,ds=0$ the result is trivial,
because we have $\gamma_j(t)=0$ for all $t\geq 0$ and each $j=1,\ldots,m$, in which case $U(t)=0$ for all 
$t\geq 0$ a.s.  

Suppose that  $\sum_{j=1}^m \int_0^\infty \gamma_j^2(s)\,ds>0$. Since $\gamma$ is square integrable 
there exists $T_0>0$ such that
\begin{equation}\label{defvposvar}
v^2(t) \geq \frac{1}{2}\sum_{j=1}^m \int_0^\infty \gamma_j^2(s)\,ds>0, \quad\text{for all $t\geq T_0$}.
\end{equation}
Let $N_0$ be the minimal integer greater that $T_0$. Let $\theta\in(0,1)$. Note that for $n\geq N_0$ 
we have that $X_n(\theta):=U(n^\theta)/v(n^\theta)$ is a standard normal random variable. Let $\epsilon>0$ and 
$n\geq N_0$. Then  
\begin{align*}
\mathbb{P}[|X_n(\theta)|\geq \sqrt{1+\epsilon}\sqrt{2\log n}]
&\leq \frac{2}{\sqrt{2\pi}}\frac{1}{\sqrt{1+\epsilon}\sqrt{2\log n}}e^{-(1+\epsilon)\log n }\\
&=\frac{2}{\sqrt{2\pi}}\frac{1}{\sqrt{1+\epsilon}\sqrt{2\log n}}\cdot \frac{1}{n^{1+\epsilon}}. 
\end{align*}
By the Borel--Cantelli lemma there exists an a.s. event $\Omega_{\epsilon,\theta}$ such that  
\[
\limsup_{n\to\infty} \frac{|X_n(\theta)|}{\sqrt{2\log n}}\leq \sqrt{1+\epsilon}, 
\quad \text{a.s. on $\Omega_{\epsilon,\theta}$}.
\]
Let $\Omega_\theta=\cap_{\epsilon\in(0,1)\cap\mathbb{Q}} \Omega_{\epsilon,\theta}$. Then $\Omega_\theta$ is an a.s. 
event. Then 
\[
\limsup_{n\to\infty} \frac{|U(n^\theta)/v(n^\theta)|}{\sqrt{2\log n}}
=
\limsup_{n\to\infty} \frac{|X_n(\theta)|}{\sqrt{2\log n}}\leq 1,
\quad \text{a.s. on $\Omega_\theta$}.
\]
Since $v^2(n^\theta)\to \int_0^\infty \sum_{j=1}^m \gamma_j^2(s)\,ds$, we have proven part (a).  

We now prove part (b). In the case when $\sum_{j=1}^m \int_0^\infty \gamma_j^2(s)\,ds=0$ the result is trivial.
Suppose instead that $\sum_{j=1}^m \int_0^\infty \gamma_j^2(s)\,ds>0$. Note that 
\[
\Cov(U(t),U(t+h))=\int_0^t \sum_{j=1}^m \gamma_j(s)\gamma_j(s+h)\,ds, \quad t\geq 0,\,h\geq 0.
\]
With $T_0>0$ defined in \eqref{defvposvar}, we let 
$t\geq T_0$, so that $v(t)>0$ and $v(t+h)\geq v(t)>0$. Then by \eqref{defvposvar}
\begin{align*}
\left|\Corr(U(t),U(t+h))\right|
&=\frac{1}{v(t)v(t+h)}\left|\int_0^t \sum_{j=1}^m \gamma_j(s)\gamma_j(s+h)\,ds \right|\\
&\leq \frac{1}{v(t)^2}\int_0^t \sum_{j=1}^m |\gamma_j(s)||\gamma_j(s+h)|\,ds \\
&\leq \frac{1}{\frac{1}{2}\sum_{j=1}^m \int_0^\infty \gamma_j^2(s)\,ds}\int_0^t \sum_{j=1}^m |\gamma_j(s)|ce^{-\alpha(s+h)}\,ds\\
&\leq e^{-\alpha h}\frac{2c}{\sum_{j=1}^m \int_0^\infty \gamma_j^2(s)\,ds}\int_0^t \sum_{j=1}^m |\gamma_j(s)|e^{-\alpha s}\,ds.
\end{align*}
Since $|\gamma_j(t)|\leq ce^{-\alpha t}$, the righthand side is finite, and moreover for all $t\geq T_0$ and $h\geq 0$ there is a $c_1>0$ independent of $t$ and $h$ such that
\[
|\Corr(U(t),U(t+h))|\leq c_1 e^{-\alpha h}.
\]

If $c_1\in(0,1]$, we can define the process $(X_n)_{n\geq 0}$ by $X_n=U(T_0+n)/v(T_0+n)$ for $n\geq 0$.
Define $\lambda:=e^{-\alpha}\in(0,1)$. Clearly $X_n$ is a standard normal random variable for each $n$. Furthermore for all $h\in\mathbb{N}$ we have
\begin{align*}
|\Cov(X_n,X_{n+h})|&=|\Cov(U(T_0+n)/v(T_0+n),U(T_0+n+h)/v(T_0+n+h))|\\
&=|\Corr(U(T_0+n),U(T_0+n+h))|\leq c_1 e^{-\alpha h}\leq e^{-\alpha h}=\lambda^h.
\end{align*}
Therefore by Lemma~\ref{theorem:discrete lower bound} we have
\[
\lim_{n\to\infty} \frac{\max_{1\leq j\leq n} U(T_0+j)/v(T_0+j)}{\sqrt{2\log n}}=1, \quad\text{a.s.}
\]
Next we have
\begin{equation*} 
\limsup_{n\to \infty}\frac{|U(T_0+n)|/v(T_0+n)}{\sqrt {2\log
n}}=\limsup_{n\to \infty}\frac{\max_{1\leq j\leq n}{|U(T_0+j)|/v(T_0+j)}}{\sqrt {2\log n}},
\end{equation*}
combining these relations gives
\[
\limsup_{n\to \infty}\frac{|U(T_0+n)|/v(T_0+n)}{\sqrt {2\log n}}\geq 1,
\quad\text{a.s.}
\]
Define $\Gamma^2=\int_0^\infty \sum_{j=1}^m \gamma_j^2(s)\,ds$. Then $v(t)\to \Gamma>0$ as $t\to\infty$.
Therefore
\begin{align*}
\limsup_{t\to \infty}\frac{|U(t)|}{\Gamma\sqrt {2\log t}}
&=\limsup_{t\to \infty}\frac{|U(t)|/v(t)}{\sqrt {2\log t}}\\
&\geq \limsup_{n\to \infty}\frac{|U(T_0+n)|/v(T_0+n)}{\sqrt {2\log (T_0+n)}}\\
&= \limsup_{n\to \infty}\frac{|U(T_0+n)|/v(T_0+n)}{\sqrt{2\log n}}
\cdot\frac{\sqrt{2\log n}}{\sqrt {2\log (T_0+n)}}.
\end{align*}
Therefore we have
\[
\limsup_{t\to \infty}\frac{|U(t)|}{\Gamma\sqrt {2\log t}}\geq 1, \quad \text{a.s.}
\]
proving part (b) in the case where $c_1\in[0,1]$.

Suppose to the contrary that $c_1>1$. Choose $N\in\mathbb{N}$ so that $c_1e^{-\alpha N}<1$.
We can define the process $(X_n)_{n\geq 0}$ by $X_n=U(T_0+Nn)/v(T_0+Nn)$ for $n\geq 0$.
Define $\lambda:=c_1e^{-\alpha N} \in(0,1)$. Clearly $X_n$ is a standard normal random variable for each $n$.
Furthermore for all $h\in\mathbb{N}$ we have
\begin{align*}
|\Cov(X_n,X_{n+h})|&=\left|\Cov\left(\frac{U(T_0+Nn)}{v(T_0+Nn)},\frac{U(T_0+Nn+Nh)}{v(T_0+Nn+Nh)}\right)\right|\\
&=|\Corr(U(T_0+Nn),U(T_0+Nn+Nh))|\leq c_1 e^{-\alpha Nh}\\
&\leq c_1^h e^{-\alpha Nh} =\lambda^h,
\end{align*}
because $c_1>1$.
Therefore by Lemma~\ref{theorem:discrete lower bound} we have
\[
\lim_{n\to\infty} \frac{\max_{1\leq j\leq n} U(T_0+Nj)/v(T_0+Nj)}{\sqrt{2\log n}}=1, \quad\text{a.s.}
\]
Next we have
\begin{equation*} 
\limsup_{n\to \infty}\frac{|U(T_0+Nn)|/v(T_0+Nn)}{\sqrt {2\log
n}}=\limsup_{n\to \infty}\frac{\max_{1\leq j\leq n}{|U(T_0+Nj)|/v(T_0+Nj)}}{\sqrt {2\log n}},
\end{equation*}
so combining these relations gives
\[
\limsup_{n\to \infty}\frac{|U(T_0+Nn)|/v(T_0+Nn)}{\sqrt {2\log n}}\geq 1,
\quad\text{a.s.}
\]
Therefore
\begin{align*}
\limsup_{t\to \infty}\frac{|U(t)|}{\Gamma\sqrt {2\log t}}
&=\limsup_{t\to \infty}\frac{|U(t)|/v(t)}{\sqrt {2\log t}}\\
&\geq \limsup_{n\to \infty}\frac{|U(T_0+Nn)|/v(T_0+Nn)}{\sqrt {2\log (T_0+Nn)}}\\
&= \limsup_{n\to \infty}\frac{|U(T_0+Nn)|/v(T_0+Nn)}{\sqrt{2\log n}}
\cdot\frac{\sqrt{2\log n}}{\sqrt {2\log (T_0+Nn)}}.
\end{align*}
Therefore we have
\[
\limsup_{t\to \infty}\frac{|U(t)|}{\Gamma\sqrt {2\log t}}\geq 1, \quad \text{a.s.}
\]
proving part (b) in the case where $c_1>1$.

In the case when $c_1\in(0,1]$ or when $c_1>1$ we have that there exist $T_0>0$ and $N\geq 1$ such 
that $X_n:=U(T_0+Nn)/v(T_0+Nn)$ defines a sequence of standard zero mean normal random variables for which 
\[
\lim_{n\to\infty} \frac{\max_{1\leq j\leq n} U(T_0+jN)}{\sqrt{2\log n}}=1, \quad \text{a.s.}
\]
Therefore a.s. we have 
\begin{align*}
\limsup_{t\to\infty} \frac{U(t)}{\sqrt{2\log t}}
&\geq 
\limsup_{n\to\infty} \frac{U(T_0+nN)}{\sqrt{2\log (T_0+Nn)}}
=\limsup_{n\to\infty} \frac{X_n v(T_0+nN)}{\sqrt{2\log (T_0+Nn)}}\\
&= \limsup_{n\to\infty} \frac{X_n}{\sqrt{2\log n}} \cdot \Gamma
= \limsup_{n\to\infty} \frac{\max_{1\leq j\leq n}X_j}{\sqrt{2\log n}} \cdot \Gamma\\
&=\lim_{n\to\infty}  \frac{\max_{1\leq j\leq n}X_j}{\sqrt{2\log n}} \cdot \Gamma=\Gamma.
\end{align*}
proving \eqref{eq.UgtGlogt}. The above argument for part (b) can be applied equally to $-U$, so that we get 
\[
\limsup_{t\to\infty} \frac{-U(t)}{\sqrt{2\log t}}\geq \Gamma, \quad \text{a.s.},
\] 
from which \eqref{eq.UltminGlogt} can be deduced. 
\end{proof}

We need one other estimate on the asymptotic behaviour of a Gaussian process.  
\begin{lemma} \label{lemma.ZtZnthsmall}
Suppose $B$ is an $m$--dimensional Brownian motion. For $j=1,\ldots,m$ suppose that $\kappa_j$ is in 
$C^1((0,\infty);\mathbb{R})$ in such a way that 
$\kappa_j\in L^2([0,\infty),\mathbb{R})$ and $\kappa_j'\in L^2([0,\infty),\mathbb{R})$. 
and define $Z=\{Z(t):t\geq 0\}$ by 
\[
Z(t)=\sum_{j=1}^m \int_0^t  \kappa_j(t-s)\,dB_j(s),\quad t\geq 0.
\]
Suppose that $\theta\in(0,1)$. Then 
\begin{equation} \label{eq.Zincsto0}
\lim_{n\to\infty} \sup_{n^\theta \leq t\leq (n+1)^\theta} |Z(t)-Z(n^\theta)|=0, \quad \text{a.s.}
\end{equation}
\end{lemma}
\begin{proof}
We note that 
\begin{align*}
Z(t)
&=\sum_{j=1}^m \int_0^t \left(\kappa_{j}(0)+\int_0^{t-s} \kappa_j'(u)\,du\right)\,dB_j(s)\\
&=\sum_{j=1}^m \kappa_j(0)B_j(t)+\sum_{j=1}^m\int_0^t \int_0^u \kappa_j'(u-s)\,dB_j(s)\,du.
\end{align*}
Hence for $t\in [n^\theta,(n+1)^\theta]$, we get
\[
Z(t)-Z(n^\theta)= \sum_{j=1}^m \kappa_j(0)\left(B_j(t)-B_j(n^\theta)\right)
+ \sum_{j=1}^m \int_{n^\theta}^t \int_0^u \kappa_j'(u-s)\,dB_j(s)\,du,
\]
which implies
\begin{equation} \label{eq.repincs}
\sup_{n^\theta\leq t\leq (n+1)^\theta}
|Z(t)-Z(n^\theta)| \leq \sum_{j=1}^m |\kappa_j(0)|\sup_{n^\theta \leq t\leq
(n+1)^\theta}|B_j(t)-B_j(n^\theta)| + \sum_{j=1}^m U_j(n),
\end{equation}
where we have defined
\[
U_j(n)=\sup_{n^\theta\leq t\leq (n+1)^\theta}\left|\int_{n^\theta}^t \int_0^u
\kappa_j'(u-s)\,dB_j(s)\,du\right|.
\]
Now if $V$ is a normal random variable with zero mean, and $p\geq 2$, there is a constant $c(p)>0$ 
such that $\mathbb{E}[|V|^p]=c(p)\mathbb{E}[V^2]^{p/2}$. Let $p\geq 2/(1-\theta)>2$. Therefore we have 
\begin{align*}
\mathbb{E}[U_j(n)^p]
&=
\mathbb{E}\left[\sup_{n^\theta\leq t\leq (n+1)^\theta}\left|\int_{n^\theta}^t \int_0^u
\kappa_j'(u-s)\,dB_j(s)\,du\right|^p\right]\\
&\leq \mathbb{E}\left[
\sup_{n^\theta\leq t\leq (n+1)^\theta} \left(\int_{n^\theta}^t \left|\int_0^u
\kappa_j'(u-s)\,dB_j(s)\right|\,du\right)^p\right]\\
&\leq \mathbb{E}\left[\sup_{n^\theta\leq t\leq (n+1)^\theta} (t-n^\theta)^{p-1}\int_{n^\theta}^t \left|\int_0^u
\kappa_j'(u-s)\,dB_j(s)\right|^p\,du\right]\\
&=  ((n+1)^\theta-n^\theta)^{p-1}\int_{n^\theta}^{(n+1)^\theta} \mathbb{E}\left|\int_0^u 
\kappa_j'(u-s)\,dB_j(s)\right|^p\,du\\
&=  ((n+1)^\theta-n^\theta)^{p-1}
\int_{n^\theta}^{(n+1)^\theta} c(p)\left(\int_0^u (\kappa_j')^2(s)\,ds\right)^{p/2}\,du.
\end{align*}
Since $\kappa_j'$ is in $L^2([0,\infty),\mathbb{R})$ we have 
\[
\mathbb{E}[U_j(n)^p]
\leq c(p)\left(\int_0^\infty (\kappa_j')^2(s)\,ds\right)^{p/2}
((n+1)^\theta-n^\theta)^{p}.
\] 
Since $p\geq 2/(1-\theta)$, we have that $\sum_{l=1}^\infty \mathbb{E}[U_j(l)^p]<+\infty$, and therefore 
by the Borel--Cantelli lemma it follows that 
\begin{equation} \label{eq.Ujto0}
\lim_{n\to\infty} U_j(n)= 0, \quad\text{a.s.}
\end{equation}

Let $\epsilon>0$. Then by the properties of a standard Brownian
motion, we have
\begin{align*}
\mathbb{P}\left[\sup_{n^\theta\leq t\leq
(n+1)^\theta}|B_j(t)-B_j(n^\theta)|>\epsilon\right]
&\leq 2\mathbb{P}\left[\sup_{0\leq t\leq (n+1)^\theta-n^\theta}B_j(t)>\epsilon\right]\\
&=2\mathbb{P}[|B_j((n+1)^\theta-n^\theta)|> \epsilon]\\
&=4\mathbb{P}\left[Z>\frac{\epsilon}{\sqrt{(n+1)^\theta-n^\theta}}\right],
\end{align*}
where $Z$ is a standard normal random variable. Since
$\{(n+1)^\theta-n^\theta\}/n^{\theta-1}\to\theta$ as
$n\to\infty$, by Mill's estimate and the Borel--Cantelli lemma,
there exists $N(\omega,\epsilon) \in \mathbb{N} $, such that for all $n>N(\epsilon)$
\begin{equation*}
\sup_{n^\theta\leq t\leq (n+1)^\theta}|B_j(t)-B_j(n^\theta)| \leq \epsilon, \quad \text{a.s. on $\Omega_\epsilon$}
\end{equation*}
Define $\Omega^\ast=\cap_{\epsilon\in (0,1)\cap\mathbb{Q}}\Omega_\epsilon$. Then $\Omega^\ast$ is an a.s. 
event and we have 
\begin{equation} \label{eq.BMincto0}
\limsup_{n\to\infty} \sup_{n^\theta \leq t\leq
(n+1)^\theta}|B_j(t)-B_j(n^\theta)|=0, \quad\text{a.s. on $\Omega^\ast$.}
\end{equation}
Taking the limit as $n\to\infty$ across both sides of \eqref{eq.repincs} and using \eqref{eq.Ujto0} 
and \eqref{eq.BMincto0} we obtain \eqref{eq.Zincsto0} as required. 
\end{proof}

\subsection{Proof of Theorem~\ref{finidimneut}}
 Let $x$ be the solution of
\eqref{eq:deterministic} and $X$ the solution of \eqref{eq:snde}.
Then by Theorem~\ref{thm.varconstneut} for $t\geq 0$ we have
$X(t)=x(t)+W(t)$. If $Z$ is defined by \eqref{def.Z}, the argument
of Theorem~\ref{thm.varconstneut} tells us that $W=Z-\rho_0\ast Z$
and $Z$ obeys \eqref{eq.varconstZkappa}. Therefore
\[
X(t)=x(t)+W(t)=x(t)+Z(t)-\int_{[0,t]}\rho_0(ds) Z(t-s), \quad t\geq
0.
\]
Let $\theta\in (0,1)$. Let $t\geq 0$ and $n$ be an integer such that
$t\in [n^\theta,(n+1)^{\theta})$. Then 
\begin{align*}
X(t)-X(n^\theta)&=x(t)-x(n^\theta)+Z(t)-Z(n^\theta)\\
&\quad\quad\quad\quad\quad\quad\quad-\left(\int_{[0,t]}
\rho_0(ds)Z(t-s)-\int_{[0,n^\theta]}\rho_0(ds)Z(n^\theta-s)\right)\\
&=x(t)-x(n^\theta)+Z(t)-Z(n^\theta)\\
&\quad\quad\quad\quad\quad\quad\quad-\bigg(\int_{[0,t]} \rho_0(ds)Z(t-s)-\int_{[0,n^\theta]}\rho_0(ds)Z(t-s)\\
&\quad\quad\quad\quad\quad\quad\quad+\int_{[0,n^\theta]}\rho_0(ds)Z(t-s)-\int_{[0,n^\theta]}
\rho_0(ds)Z(n^\theta-s)\bigg)\\
&=x(t)-x(n^\theta)+Z(t)-Z(n^\theta)-\int_{[n^\theta,t]}\rho_0(ds)Z(t-s)\\
&\quad\quad\quad\quad\quad\quad\quad-\int_{[0,n^\theta]}\rho_0(ds)\bigg(Z(t-s)-Z(n^\theta-s)\bigg).
\end{align*}
Therefore
\begin{multline}\label{3terms}
\sup_{n^\theta\leq t\leq (n+1)^\theta}|X(t)-X(n^\theta)|
\\
\leq \sup_{n^\theta\leq t\leq (n+1)^\theta}|x(t)-x(n^{\theta})|+
\sup_{n^\theta\leq t\leq (n+1)^\theta}|Z(t)-Z(n^\theta)|
\\+\sup_{n^\theta\leq t\leq
(n+1)^\theta}\bigg|\int_{[n^\theta,t]} \rho_0(ds)Z(t-s)\bigg|\\
+\sup_{n^\theta\leq t\leq
(n+1)^\theta}\bigg|\int_{[0,n^{\theta}]}\rho_0(ds)\big(Z(t-s)-Z(n^{\theta}-s)\big)\bigg|.
\end{multline}
We now consider each of the four terms on the right-hand side of
\eqref{3terms} in turn. It is easy to see that
\[
\lim_{n\to\infty}\sup_{n^\theta\leq t\leq
(n+1)^\theta}|x(t)-x(n^\theta)|=0.
\]
By applying Lemma~\ref{lemma.ZtZnthsmall} component by component it follows that 
\begin{equation}\label{Zmesh}
\lim_{n\to\infty}\sup_{n^\theta\leq t\leq
(n+1)^\theta}|Z(t)-Z(n^\theta)|=0,\quad\text{a.s.}
\end{equation}
For the third term,
\begin{align*}
\sup_{n^{\theta}\leq t\leq (n+1)^\theta}\bigg|\int_{[n^\theta,t]}
\rho_0(ds)Z(t-s)\bigg|
&\leq \sup_{n^\theta\leq t\leq (n+1)^\theta} \int_{[n^\theta,t]} |\rho_0|(ds) |Z(t-s)|\\
&\leq \sup_{n^\theta\leq t\leq (n+1)^\theta}\sup_{n^\theta\leq s\leq
t}|Z(t-s)|\cdot \int_{[n^\theta,\infty)} |\rho_0|(ds)
\\&=\sup_{n^\theta\leq s\leq t\leq (n+1)^\theta}|Z(t-s)|\cdot \int_{[0,\infty)} |\rho_0|(ds)\\
&=\sup_{0\leq u\leq (n+1)^\theta-n^\theta}
|Z(u)|\cdot\int_{[0,\infty)}|\rho_0|(ds),
\end{align*}
which implies
\begin{multline}\label{2ndterm}
\limsup_{n\to\infty}
\sup_{n^{\theta}\leq t\leq (n+1)^\theta}\bigg|\int_{[n^\theta,t]}
\rho_0(ds)Z(t-s)\bigg|
\\
\leq 
\limsup_{n\to\infty}
\sup_{0\leq u\leq (n+1)^\theta-n^\theta}
|Z(u)|\cdot\int_{[0,\infty)}|\rho_0|(ds)
=|Z(0)|\cdot\int_{[0,\infty)}|\rho_0|(ds)=0,
\quad\text{a.s.}
\end{multline}
For the last term on the right-hand side of \eqref{3terms}, we note
that for $t\geq 0$, by \eqref{eq.varconstZkappa}
\begin{align*}
Z(t)&=\int_0^t(\kappa(0)+\int_0^{t-s}\kappa'(v)\,dv)\Sigma\,dB(s)\\
&=\kappa(0)\Sigma B(t)+\int_0^t\int_0^{t-s}\kappa'(v)\,dv\Sigma\,dB(s)\\
&=\Sigma B(t)+\int_0^t\int_s^t\kappa'(u-s)\Sigma\,du\,dB(s)\\
&=\Sigma B(t)+\int_0^t\int_0^u\kappa'(u-s)\Sigma\,dB(s)\,du.
\end{align*}
So for $n^\theta\leq t\leq (n+1)^\theta$,
\begin{multline*}
Z(t-s)-Z(n^\theta-s)=\Sigma \bigg(B(t-s)-B(n^\theta-s)\bigg)
+\int_{n^\theta-s}^{t-s}\int_0^u\kappa'(u-v)\Sigma\,dB(v)\,du.
\end{multline*}
Hence
\begin{multline}\label{Zmeshts}
\sup_{n^\theta\leq t\leq (n+1)^\theta}\bigg|\int_{[0,n^\theta]}\rho_0(ds)(Z(t-s)-Z(n^\theta-s))\bigg|\\
\leq |\Sigma|\sup_{n^\theta\leq t\leq (n+1)^\theta}\int_{[0,n^\theta]}|\rho_0|(ds)|B(t-s)-B(n^\theta-s)|\\
+\sup_{n^\theta\leq t\leq
(n+1)^\theta}\int_{[0,n^\theta]}|\rho_0|(ds)\bigg|\int_{n^\theta-s}^{t-s}\int_0^u\kappa'(u-v)\Sigma\,dB(v)\,du\bigg|.
\end{multline}
For the first term on the right-hand side of \eqref{Zmeshts}, for
some  $p_\theta>1$ and $q_\theta>1$ such that $1/p_\theta+
1/q_\theta=1$,
\begin{align*}
&\mathbb{E}\left[\left(\sup_{n^\theta\leq t\leq (n+1)^\theta}\int_{[0,n^\theta]}|\rho_0|(ds)|B(t-s)-B(n^\theta-s)|\right)^{p_\theta}\right]\\
&\leq \mathbb{E}\bigg[\sup_{n^\theta\leq t\leq
(n+1)^\theta}\left(\int_{[0,n^\theta]}|\rho_0|(ds)\right)^{\frac{p_\theta}{q_\theta}}
\left(\int_{[0,n^\theta]}|\rho_0|(ds)|B(t-s)-B(n^\theta-s)|^{p_\theta}\right)\bigg]\\
&\leq
\left(\int_{[0,\infty)}|\rho_0|(ds)\right)^{\frac{p_\theta}{q_\theta}}\mathbb{E}
\left[\sup_{n^\theta\leq t\leq (n+1)^\theta}\left(\int_{[0,n^\theta]}|\rho_0|(ds)|B(t-s)-B(n^\theta-s)|^{p_\theta}\right)\right]\\
&\leq
\left(\int_{[0,\infty)}|\rho_0|(ds)\right)^{\frac{p_\theta}{q_\theta}}\mathbb{E}
\left[\int_{[0,n^\theta]}|\rho_0|(ds)\sup_{n^\theta\leq t\leq (n+1)^\theta}|B(t-s)-B(n^\theta-s)|^{p_\theta}\right]\\
&\leq
\left(\int_{[0,\infty)}|\rho_0|(ds)\right)^{\frac{p_\theta}{q_\theta}}\int_{[0,n^\theta]}|\rho_0|(ds)\,\mathbb{E}
\left[\sup_{n^\theta-s\leq u\leq (n+1)^\theta-s}|B(u)-B(n^\theta-s)|^{p_\theta}\right]\\
&\leq
\left(\int_{[0,\infty)}|\rho_0|(ds)\right)^{\frac{p_\theta}{q_\theta}}\int_{[0,n^\theta]}|\rho_0|(ds)\,\mathbb{E}
\left[\sup_{n^\theta-s\leq u\leq (n+1)^\theta-s}\bigg|\int_{n^\theta-s}^u\,\,dB(v)\bigg|^{p_\theta}\right]\\
&\leq
\left(\int_{[0,\infty)}|\rho_0|(ds)\right)^{\frac{p_\theta}{q_\theta}}\int_{[0,n^\theta]}|\rho_0|(ds)\left(\frac{32}{p_\theta}\right)
^{\frac{p_\theta}
{2}}\mathbb{E}\left[((n+1)^\theta-n^\theta)^{\frac{p_\theta}{2}}\right]\\
&\leq \left(\frac{32}{p_\theta}\right)^{\frac{p_\theta}
{2}}\left(\int_{[0,\infty)}
|\rho_0|(ds)\right)^{\frac{p_\theta+q_\theta}{q_\theta}}[(n+1)^\theta-n^\theta]^{\frac{p_\theta}{2}},
\end{align*}
where we have used the H\"{o}lder inequality and
Burkholder-Davis-Gundy inequality in the second and penultimate
lines respectively. Hence by the Chebyshev inequality
\begin{align*}
\mathbb{P}\bigg[\sup_{n^\theta\leq t\leq (n+1)^\theta}&\int_{[0,n^\theta]} |\rho_0|(ds)|B(t-s)-B(n^\theta-s)|>1\bigg]\\
&\leq \mathbb{E}\bigg[\left(\sup_{n^\theta\leq t\leq (n+1)^\theta}\int_{[0,n^\theta]} |\rho_0|(ds)|B(t-s)-B(n^\theta-s)|\right)^{p_\theta}\bigg]\\
&\leq \left(\frac{32}{p_\theta}\right)^{\frac{p_\theta}
{2}}\left(\int_{[0,\infty)}
|\rho_0|(ds)\right)^{\frac{p_\theta+q_\theta}{q_\theta}}[(n+1)^\theta-n^\theta]^{\frac{p_\theta}{2}}.
\end{align*}
Now since
$\lim_{n\to\infty}[(n+1)^\theta-n^\theta]/n^{(\theta-1)}=\theta$, if
we choose $p_\theta=4/(1-\theta)>1$, then by the Borel-Cantelli
lemma, we get
\begin{equation}\label{bound1}
\limsup_{n\to\infty}\sup_{n^\theta\leq t\leq
(n+1)^\theta}\int_{[0,n^\theta]}|\rho_0|(ds)|B(t-s)-B(n^\theta-s)|\leq
1\quad\text{a.s.}
\end{equation}
For the second term on the right-hand side of \eqref{Zmeshts},
define $I(u):=\int_0^u\kappa'(u-v)\Sigma\, dB(v)$ and
$H_n(s):=\int_{n^\theta-s}^{(n+1)^\theta-s} |I(u)|\,du$. Then
\begin{align*}
A_n&:=
\sup_{n^\theta\leq t\leq (n+1)^\theta}\int_{[0,n^\theta]}|\rho_0|(ds)\bigg|\int_{n^\theta-s}^{t-s}\int_0^u\kappa'(u-v)\Sigma\,dB(v)\,du\bigg|\\
\nonumber
&\leq \sup_{n^\theta\leq t\leq (n+1)^\theta}\int_{[0,n^\theta]}|\rho_0|(ds)\int_{n^\theta-s}^{t-s}\bigg|\int_0^u\kappa'(u-v)\Sigma\,dB(v)\bigg|\,du\\
\nonumber &\leq \int_{[0,n^\theta]}|\rho_0|(ds) H_n(s).
\end{align*}
Therefore if $p_\theta>1$ and $q_\theta>1$ are such that
$1/p_\theta+1/q_\theta=1$, then by H\"{o}lder's inequality we have
\begin{align*}
A_n^{p_\theta}&\leq \left(\int_{[0,n^\theta]}|\rho_0|(ds) H_n(s)\right)^{p_\theta}\\
&\leq \left(\int_{[0,n^\theta]}|\rho_0|(ds) \right)^{\frac{p_\theta}{q_\theta}}  \int_{[0,n^{\theta}]} |\rho_0|(ds)H_n(s)^{p_\theta}\\
&\leq \left(\int_{[0,n^\theta]}|\rho_0|(ds)
\right)^{\frac{p_\theta}{q_\theta}}\\
&\qquad\qquad \times\int_{[0,n^{\theta}]} |\rho_0|(ds)
\left((n+1)^{\theta}-n^{\theta}\right)^{p_\theta-1}
\int_{n^\theta-s}^{(n+1)^\theta-s}|I(u)|^{p_\theta}\,du.
\end{align*}
Now as $\kappa'\in L^2$, we have that
\[
\mathbb{E}[|I(u)|^2]= \int_0^u |\kappa'(s)\Sigma|^2_F\,ds, \quad
u\geq 0.
\]
Therefore there exists $K(p)>0$ such that $\mathbb{E}[|I(u)|^p] \leq
K(p)$ for all $u\geq 0$. Therefore
\begin{align*}
\mathbb{E}[A_n^{p_\theta}]
&\leq \left((n+1)^{\theta}-n^{\theta}\right)^{p_\theta-1} \left(\int_{[0,\infty)} |\rho_0|(ds) \right)^{p_\theta/q_\theta} \\
&\qquad\times
\int_{[0,n^{\theta}]} |\rho_0|(ds) \int_{n^\theta-s}^{(n+1)^\theta-s}\mathbb{E}[|I(u)|^{p_\theta}]\,du\\
&\leq \left((n+1)^{\theta}-n^{\theta}\right)^{p_\theta-1} \left(\int_{[0,\infty)} |\rho_0|(ds) \right)^{p_\theta/q_\theta} \\
&\qquad\times
\int_{[0,n^{\theta}]} |\rho_0|(ds) ((n+1)^\theta-n^\theta) K(p_\theta)\\
&\leq K(p_\theta) \left((n+1)^{\theta}-n^{\theta}\right)^{p_\theta}
\left(\int_{[0,\infty)} |\rho_0|(ds) \right)^{p_\theta/q_\theta+1}.
\end{align*}
Therefore we have
\begin{equation} \label{eq.estPAn}
\mathbb{P}[A_n>1] \leq \mathbb{E}[A_n^{p_\theta}] \leq K(p_\theta)
\left((n+1)^{\theta}-n^{\theta}\right)^{p_\theta}
\left(\int_{[0,\infty)} |\rho_0|(ds) \right)^{p_\theta/q_\theta+1}.
\end{equation}
Let $p_\theta=2/(1-\theta)$; then $p_\theta>2$. Then the righthand
side of \eqref{eq.estPAn} is summable in $n$, because
$\{(n+1)^\theta-n^\theta\}/n^{\theta-1} \to \theta$ as $n\to\infty$,
so by the Borel--Cantelli Lemma we have that
\begin{multline}  \label{3rdterm}
\limsup_{n\to\infty} \sup_{n^\theta\leq t\leq
(n+1)^\theta}\int_{[0,n^\theta]}|\rho_0|(ds)\bigg|\int_{n^\theta-s}^{t-s}\int_0^u\kappa'(u-v)\Sigma\,dB(v)\,du\bigg|
\\= \limsup_{n\to\infty} A_n \leq 1, \quad \text{a.s.}
\end{multline}
Combining \eqref{Zmeshts}, \eqref{bound1} and \eqref{3rdterm}, it
follows that
\begin{equation}\label{Zineq}
\limsup_{n\to\infty}\frac{\sup_{n^\theta\leq t\leq
(n+1)^\theta}\bigg|\int_{[0,n^\theta]}
\rho_0(ds)(Z(t-s)-Z(n^\theta-s))\bigg|}
{\sqrt{2\log{n}}}=0,\quad\text{a.s.}
\end{equation}
Gathering the results 
\eqref{3terms}, \eqref{Zmesh}, \eqref{2ndterm} and \eqref{Zineq}, we
obtain
\begin{equation} \label{eq.vectorerror}
\limsup_{n\to\infty} \frac{\sup_{n^\theta\leq t\leq
(n+1)^\theta}|X(t)-X(n^\theta)|}{\sqrt{2\log n}}=0, \quad\text{a.s.}
\end{equation}

Next we consider each component.
%
%
We have $W(t)=X(t)-x(t)$
\[
W(t):=\int_0^t \rho(t-s)\Sigma\,dB(s), \quad t\geq 0.
\]
Notice that $W(t)\in \mathbb{R}^d$ for each $t\geq 0$. Also
$W(t)=\int_0^t \theta(t-s)\,dB(s)$, $t\geq 0$, where
$\theta(t)=\rho(t)\Sigma$ is a $d\times m$--matrix valued function
in which each entry must obey $|\theta_{ij}(t)|\leq C
e^{-v_0(\mu,\nu)t/2}$, $t\geq 0$ for some $C>0$. Hence
$W_i(t):=\langle W(t),\mathbf{e}_i\rangle$ obeys
\[
W_i(t)=\sum_{j=1}^m \int_0^t \theta_{ij}(t-s)\,dB_j(s), \quad t\geq
0.
\]
Define $\theta_i(t)\geq 0$ with $\theta_i^2(t)=\sum_{j=1}^m
\theta_{ij}^2(t)$, $t\geq 0$. Then $W_i(t)$ is normally distributed
with mean zero and variance $v_i(t)=\int_0^t \theta_i^2(s)\,ds$.
Since $\theta_i\in L^2(0,\infty)$, we have that $v_i(t)\to
\int_0^\infty \theta_i^2(s)\,ds=\int_0^\infty
\sum_{j=1}^m\theta_{ij}^2(t)\,dt=:\sigma_i^2$ as $t\to\infty$.
Moreover $|\theta_i(t)|\leq Cm e^{-v_0(\nu)t/2}$, $t\geq 0$. By part
(b) of Lemma~\ref{lemma.gaussproc} we have
\begin{equation} \label{mr1 lowbouifindim}
\limsup_{t\to\infty} \frac{|W_i(t)|}{\sqrt{2\log t}}\geq \sigma_i,
\quad \limsup_{t\to\infty} \frac{W_i(t)}{\sqrt{2\log t}}\geq \sigma_i, 
\quad \liminf_{t\to\infty} \frac{W_i(t)}{\sqrt{2\log t}}\leq -\sigma_i, 
\quad \text{a.s.}
\end{equation}
We now wish to prove
\begin{equation} \label{mr1 upbouifindim}
\limsup_{t\to\infty} \frac{W_i(t)}{\sqrt{2\log t}}
\leq \limsup_{t\to\infty} \frac{|W_i(t)|}{\sqrt{2\log t}}\leq \sigma_i,
\quad \text{a.s.}
\end{equation}
We first note for each $\theta>0$ that part (a) of
Lemma~\ref{lemma.gaussproc} yields the estimate
\begin{equation}\label{ubdiscfindimi}
\limsup_{n\to \infty} \frac{|W_i(n^\theta)|}{\sqrt{2\log
(n^\theta)}}\leq \sqrt{\frac{\sigma_i^2}{\theta}}, \quad\text{a.s.}
\end{equation}
Define $X_i(t)=\langle X(t),\mathbf{e}_i\rangle$ for $t\geq 0$.
Using \eqref{ubdiscfindimi}, the fact that $x(t)\to 0$ as
$t\to\infty$, and the fact that
\[
W_i(t)=W_i(n^\theta)+x_i(n^\theta)-x_i(t) + X_i(t)-X_i(n^\theta)
\]
we may use \eqref{eq.vectorerror} to obtain
\[
\limsup_{n \to \infty} \sup_{n^{\theta} \leq t \leq (n+1)^{\theta}}
\frac{|W_i(t)|}{\sqrt {2 \log t}} \leq
\sqrt{\frac{\sigma_i^2}{\theta}}, \quad \text{a.s.},
\]
which implies
\[
\limsup_{t \to \infty} \frac{|W_i(t)|}{\sqrt {2 \log t}} \leq
\sqrt{\frac{\sigma_i^2}{\theta}}, \quad \text{a.s.}
\]
Letting $\theta\to 1$ through the rational numbers implies
\eqref{mr1 upbouifindim}. Therefore by \eqref{mr1 lowbouifindim} and \eqref{mr1 upbouifindim} we have 
\[
\limsup_{t\to\infty} \frac{W_i(t)}{\sqrt{2\log t}}=\sigma_i, 
\quad 
\limsup_{t\to\infty} \frac{|W_i(t)|}{\sqrt{2\log t}}=\sigma_i, 
\quad\text{a.s.}
\]
It is a consequence of \eqref{mr1 upbouifindim} that 
\[
-\liminf_{t\to\infty} \frac{W_i(t)}{\sqrt{2\log t}}=
\limsup_{t\to\infty} \frac{-W_i(t)}{\sqrt{2\log t}}
\leq  
\limsup_{t\to\infty} \frac{|W_i(t)|}{\sqrt{2\log t}}
\leq 
\sigma_i,
\quad \text{a.s.}
\]
Combining this with the third inequality in \eqref{mr1 lowbouifindim} we obtain 
\[
\liminf_{t\to\infty} \frac{W_i(t)}{\sqrt{2\log t}}=-\sigma_i,\quad \text{a.s.}
\]
From all these estimates, and recalling that $x(t)\to 0$ as
$t\to\infty$, we obtain \eqref{eq:finidimlimsupcompneut} as required. 

To prove \eqref{eq:finidimlimsupneut}, note that there is an
$i^\ast\in \{1,\ldots,d\}$ such that $\sigma_{i^\ast}=\max_{1\leq
i\leq d} \sigma_i$. Next, note for each $t\geq 0$ that
\[
\max_{1\leq i\leq d}
|X_i(t)|=\max(|X_1(t)|,|X_2(t)|,\ldots,|X_{i^\ast}(t)|,\ldots,|X_d(t)|)\geq
|X_{i^\ast}(t)|.
\]
Hence
\begin{equation} \label{eq:findimlowermax}
\limsup_{t\to \infty}\frac{\max_{1\leq i\leq d}|X_i(t)|}{\sqrt{2\log
t}} \geq \limsup_{t\to \infty}\frac{|X_{i^\ast}(t)|}{\sqrt{2\log
t}}=\sigma_{i^\ast}= \max_{i=1,\ldots,d}\sigma_i,\quad \text{a.s.}
\end{equation}
Let $p$ be an integer greater than unity. Note that $\max_{1\leq
i\leq d} |x_i|\leq (\sum_{i=1}^d |x_i|^p)^{1/p}$, so we have
\begin{align*}
\left(\limsup_{t\to\infty} \frac{\max_{1\leq i\leq d}
|X_i(t)|}{\sqrt{2\log t}}\right)^p &= \limsup_{t\to\infty}
\frac{\left(\max_{1\leq i\leq d} |X_i(t)|\right)^p}{(\sqrt{2\log t})^p}\\
&\leq \limsup_{t\to\infty} \frac{\sum_{i=1}^d |X_i(t)|^p}{(\sqrt{2\log t})^p}\\
&\leq \sum_{i=1}^d\limsup_{t\to\infty}  \frac{|X_i(t)|^p}{(\sqrt{2\log t})^p}\\
&=\sum_{i=1}^d\left(\limsup_{t\to\infty} \frac{|X_i(t)|}{\sqrt{2\log
t}}\right)^p=\sum_{i=1}^d \sigma_i^p.
\end{align*}
Hence
\[
\limsup_{t\to\infty} \frac{\max_{1\leq i\leq d}
|X_i(t)|}{\sqrt{2\log t}} \leq \left(\sum_{i=1}^d
\sigma_i^p\right)^{1/p}, \quad \text{a.s.}
\]
Letting $p\to\infty$ through the natural numbers yields
\begin{equation}\label{eq:findimuppermax}
\limsup_{t\to\infty} \frac{\max_{1\leq i\leq d}
|X_i(t)|}{\sqrt{2\log t}} \leq \max_{1\leq i\leq d} \sigma_i, \quad
\text{a.s.},
\end{equation}
since $\left(\sum_{i=1}^d \sigma_i^p\right)^{1/p}\to \max_{1\leq
i\leq d} \sigma_i$ as $p\to\infty$. Combining
\eqref{eq:findimlowermax} and \eqref{eq:findimuppermax} yields
\eqref{eq:finidimlimsupneut}.

\section{Proof of Theorem~\ref{thm.nonlinrandwalkneut}}
Let $X$ be the solution of \eqref{eq.nonlinearsfde}.
Suppose that $Y$ obeys
\begin{equation} \label{eq.linearneutralstoch}
d(Y(t)-D(Y_t)) = L(Y_t)\,dt + \Sigma \,dB(t), \quad t\geq 0; \quad
Y_0=\phi.
\end{equation}
Define for $t\geq -\tau$ the process $Z(t):=X(t)-Y(t)$. Since $D$
and $L$ are linear, $Z$ obeys
\begin{multline} \label{eq.odeZ}
\frac{d}{dt} \left( Z(t)-D(Z_t)-N_1(t,Z_t+Y_t)\right)\\
= L(Z_t) +
N_2(t,Z_t+Y_t), \quad t>0; \quad Z_0\equiv0.
\end{multline}
Define $U=\{U(t):t\geq 0\}$ by
\begin{equation} \label{def.Uneut}
U(t)=Z(t)-D(Z_t)-N_1(t,Z_t+Y_t), \quad t\geq 0.
\end{equation}
By the definition of $\mu_+$ and the fact that $Z(t)=0$ for all
$t\in [-\tau,0]$, we have
\[
U(t)=Z(t)-\int_{[0,t]} \mu_+(ds) Z(t-s) - N_1(t,Z_t+Y_t), \quad
t\geq 0,
\]
so $Z-\mu_+\ast Z=U+N_1$. Recall that $\rho_0$ defined by
\eqref{integralre} is in $M([0,\infty);\mathbb{R}^{d\times d})$
because $\mu$ obeys \eqref{integresolventcondition}. Therefore
$Z=N_1+U-\rho_0\ast (N_1+U)$ or
\begin{multline} \label{eq.ZintermsU}
Z(t)=N_1(t,Z_t+Y_t)+U(t)\\
-\int_{[0,t]}
\rho_0(ds)\left(N_1(t-s,Y_{t-s}+Z_{t-s}))+U(t-s)\right), \quad t\geq
0.
\end{multline}
By \eqref{eq.odeZ}, \eqref{def.Uneut} we have $U'(t)=L(Z_t)+N_2(t,Z_t+Y_t)$ for
$t\geq 0$, so by the definition of $\nu_+$ we have
$U'(t)=\int_{[0,t]} \nu_+(ds)Z(t-s)+N_2(t,Z_t+Y_t)$. Therefore
\begin{equation} \label{eq.odeU}
U'(t)=(\nu_+\ast Z)(t)+N_2(t,Z_t+Y_t), \quad t>0.
\end{equation}
By \eqref{eq.mu0zero} we have $U(0)=-N_1(0,Y_0)=-N_1(0,\phi)$. Thus by \eqref{eq.ZintermsU} and the definition of $\beta$ from \eqref{def.beta},
we get
\begin{align}
U'(t)
&=\biggl\{\nu_+ \ast[N_1+U - \rho_0\ast \left(N_1+U\right)]\biggr\}(t)+N_2(t,Z_t+Y_t)\nonumber\\
\label{eq.UVolterradiff}
&=[(\nu_+ - \nu_+\ast \rho_0)\ast N_1](t)
+N_2(t,Z_t+Y_t) + [\beta\ast U](t).
\end{align}
Recall that $\kappa$ is the differential resolvent defined by \eqref{eq.kappadiffresolv}. Therefore, we also have
\begin{equation} \label{eq.kappaVoldiffb}
\kappa'(t)=\left(\kappa\ast\beta)\right)(t), \quad t>0.
\end{equation}
Now by \eqref{eq.UVolterradiff} and \eqref{eq.kappadiffresolv} and the
fact that $U(0)=-N_1(0,\phi)$ we have
\begin{multline*}
U(t)=-\kappa(t)N_1(0,\phi)+\kappa\ast [(\nu_+ - \nu_+\ast \rho_0)\ast N_1](t) \\
+ \int_0^t \kappa(t-s)N_2(s,Z_s+Y_s)\,ds, \quad t\geq 0.
\end{multline*}
By \eqref{eq.kappaVoldiffb} this implies
\begin{equation*}
U(t)=-\kappa(t)N_1(0,\phi)+\left(\kappa'\ast N_1\right)(t)+ \int_0^t
\kappa(t-s)N_2(s,Z_s+Y_s)\,ds, \quad t\geq 0.
\end{equation*}
Let $\kappa_1:=\kappa'$. Then $\kappa_1=\nu_+\ast \rho$. Since
$\rho\in L^1([0,\infty);\mathbb{R}^{d\times d})$, $\nu_+\in
M([0,\infty);\mathbb{R}^{d\times d})$, we have that $\kappa_1 \in
L^1([0,\infty);\mathbb{R}^{d\times d})$. Therefore with
$N_2(t):=N_2(t,Z_t+Y_t)$ we have
\begin{equation*}
U(t)=-\kappa(t)N_1(0,\phi)+ (\kappa_1\ast N_1)(t) + (\kappa\ast
N_2)(t), \quad t\geq 0.
\end{equation*}
Inserting this into \eqref{eq.ZintermsU} we get
\begin{multline*}
Z(t)=N_1(t)-\kappa(t)N_1(0,\phi)+ (\kappa_1\ast N_1)(t) +
(\kappa\ast N_2)(t) -(\rho_0 \ast N_1)(t)
\\- \left\{\rho_0 \ast [-\kappa N_1(0,\phi)+ \kappa_1\ast N_1 + \kappa\ast N_2]\right\}(t), \quad t\geq 0.
\end{multline*}
This gives
\begin{multline*}
Z(t)=N_1(t)-\left\{\kappa(t)-  (\rho_0 \ast \kappa)(t) \right\}
N_1(0,\phi)+ \left\{(\kappa  - \rho_0 \ast \kappa)\ast
N_2\right\}(t)
\\+(\kappa_1\ast N_1)(t) -(\rho_0 \ast N_1)(t) - (\rho_0 \ast \kappa_1\ast N_1)(t), \quad t\geq 0.
\end{multline*}
By \eqref{rhoinkappa} we have $\rho=\kappa-\rho_0\ast\kappa$, so if
we define $\kappa_2 \in M([0,\infty);\mathbb{R}^{d\times d})$
\[
\kappa_2=\kappa_1 -\rho_0 - (\rho_0\ast \kappa_1)
\]
we have
\begin{equation} \label{eq.ZVolint}
Z(t)=N_1(t)-\rho(t) N_1(0,\phi)+ (\rho \ast N_2)(t)+(\kappa_2 \ast
N_1)(t), \quad t\geq 0.
\end{equation}
$\kappa_2$ is guaranteed to be in $M([0,\infty);\mathbb{R}^{d\times
d})$ because $\kappa_1 \in L^1([0,\infty);\mathbb{R}^{d\times d})$
and $\rho_0 \in M([0,\infty);\mathbb{R}^{d\times d})$. Since $N_1$
and $N_2$ are bounded by maximum functionals of $Z + Y$, the
asymptotic behaviour of $Y$ is known, and the convolution
``kernels'' $\kappa_2$ and $\rho$ are finite on the right hand side,
we may treat \eqref{eq.ZVolint} pathwise as a Volterra integral
equation.

By \eqref{eq.nestimate}, for any $\phi\in C([-\tau,0];\mathbb{R}^d)$
and $\varepsilon>0$, there exists $L(\varepsilon)>0$ such that
\begin{multline*}
|N_1(t,\phi_t)|\vee |N_1(t,\phi_t)|_2\leq L(\varepsilon)+\varepsilon \max_{-\tau\leq s\leq 0} |\phi(s)|_2, \\
\text{ for all $(t,\phi)\in (0,\infty)\times
C([-\tau,\infty);\mathbb{R}^d)$}.
\end{multline*}
Define
\begin{equation} \label{def.c}
c=1+ \int_0^\infty |\rho(s)|_2 \,ds +  \int_{[0,\infty)}
|\kappa_2|(ds)
\end{equation}
Choose $\varepsilon>0$ so small that $\varepsilon c < 1/2$.
Therefore by \eqref{eq.ZVolint} for $t\geq 0$
\begin{multline*}
|Z(t)|_2\leq  L(\varepsilon)+\varepsilon\sup_{t-\tau\leq u\leq t}|Y(u)+Z(u)|_2 + |\rho(t)||N_1(0,\phi)|\\
+ \int_0^t |\rho(t-s)|_2 \left\{L(\varepsilon)+\varepsilon\sup_{s-\tau\leq u\leq s}|Y(u)+Z(u)|_2 \right\}\,ds \\
+ \int_{[0,t]} |\kappa_2|(ds)
\left\{L(\varepsilon)+\varepsilon\sup_{t-s-\tau\leq u\leq
t-s}|Y(u)+Z(u)|_2 \right\}.
\end{multline*}
By Lemma~\ref{lemmaneutralequ}, since $\rho\in
L^1([0,\infty);\mathbb{R}^{d\times d})$ we have that $\rho(t)\to 0$
as $t\to\infty$. Therefore there exists $T_1>0$ such that
$|\rho(t)|\leq 1$ for all $t\geq T_1$. Hence for $t\geq T_1$, using
\eqref{def.c}, we have
\begin{multline*}
|Z(t)|_2\leq  |N_1(0,\phi)| + L(\varepsilon)c +  \varepsilon \sup_{t-\tau\leq u\leq t} \left(|Y(u)|_2+|Z(u)|_2\right)\\
+ \varepsilon \int_0^t |\rho(t-s)|_2\sup_{s-\tau\leq u\leq s} \left(|Y(u)|_2+|Z(u)|_2\right) \,ds \\
+ \varepsilon \int_{[0,t]} |\kappa_2|(ds)\sup_{t-s-\tau\leq u\leq
t-s} \left(|Y(u)|_2+|Z(u)|_2\right).
\end{multline*}
Define $L_2(\varepsilon):= |N_1(0,\phi)| + L(\varepsilon)c$ and
\begin{multline}  \label{def.fepsilon}
f_\varepsilon(t)= L_2(\varepsilon) + \varepsilon \int_0^t
|\rho(t-s)|_2\sup_{s-\tau\leq u\leq s}|Y(u)|_2
\\+ \varepsilon \int_{[0,t]} |\kappa_2|(ds) \sup_{t-s-\tau\leq u\leq t-s}|Y(u)|_2
+ \varepsilon \sup_{t-\tau\leq u\leq t} |Y(u)|_2.
\end{multline}
Hence for $t\geq T_1$ we have
\begin{multline*}
|Z(t)|_2\leq  f_\varepsilon(t) + \varepsilon \int_0^t
|\rho(t-s)|_2\sup_{s-\tau\leq u\leq s}  |Z(u)|_2 \,ds
\\+ \varepsilon \int_{[0,t]} |\kappa_2|(ds)\sup_{t-s-\tau\leq u\leq t-s}|Z(u)|_2  +\varepsilon \sup_{t-\tau\leq u\leq t}|Z(u)|_2.
\end{multline*}
Now by \eqref{def.c} we have
\begin{equation*}
|Z(t)|_2\leq  f_\varepsilon(t) + \varepsilon c \sup_{-\tau\leq u\leq
t}  |Z(u)|_2
, \quad t\geq T_1.
\end{equation*}
By \eqref{def.c} and \eqref{def.fepsilon} we have
\[
f_\varepsilon(t)\leq  L_2(\varepsilon) + \varepsilon c
\sup_{-\tau\leq u\leq t} |Y(u)|_2.
\]
Since $Z(u)=0$ for all $u\in [-\tau,0]$, and $Y(u)=\phi(u)$ for all
$u\in [-\tau,0]$, for $t\geq T_1$ we have
\begin{equation*}
|Z(t)|_2 
\leq L_2(\varepsilon) + c\varepsilon \sup_{-\tau\leq u\leq 0}
|\phi(u)|_2 +c\varepsilon \sup_{0\leq u\leq t} |Y(u)|_2 +
c\varepsilon \sup_{0\leq u\leq t} |Z(u)|_2.
\end{equation*}
Define $L_3(\varepsilon)=L_2(\varepsilon) + c\varepsilon
\sup_{-\tau\leq u\leq 0} |\phi(u)|_2$. Then
\begin{equation*}
|Z(t)|_2\leq L_3(\varepsilon) +c\varepsilon \sup_{0\leq u\leq t}
|Y(u)|_2 + c\varepsilon \sup_{0\leq u\leq t} |Z(u)|_2, \quad t\geq
T_1.
\end{equation*}
Now for $\omega\in \Omega$ define
$L_4(\varepsilon,\omega)=L_3(\varepsilon)\vee \max_{0\leq s\leq T_1}
|Z(s,\omega)|_2$. Then for $t\in [0,T_1]$ we have
$|Z(t,\omega)|_2\leq L_4(\varepsilon,\omega)$ and so
\begin{equation*}
|Z(t,\omega)|_2\leq L_4(\varepsilon,\omega) +c\varepsilon
\sup_{0\leq u\leq t} |Y(u,\omega)|_2 + c\varepsilon \sup_{0\leq
u\leq t} |Z(u,\omega)|_2, \quad t\geq 0.
\end{equation*}
Define $Z^\ast(T)=\max_{0\leq s\leq T} |Z(s)|_2$ for any $T\geq 0$.
Therefore
\begin{equation*}
Z^\ast(T,\omega) \leq L_4(\varepsilon,\omega) +c\varepsilon
\sup_{0\leq u\leq T} |Y(u,\omega)|_2 + c\varepsilon
Z^\ast(T,\omega), \quad T\geq 0.
\end{equation*}
Since $\varepsilon c<1/2$ we have
\begin{equation} \label{eq.ZastFinal}
Z^\ast(T,\omega) \leq 2L_4(\varepsilon,\omega) +2c\varepsilon
\sup_{0\leq u\leq T} |Y(u,\omega)|_2, \quad T\geq 0.
\end{equation}
Define $Y^\ast(T)=\sup_{0\leq u\leq T} |Y(u)|_2$ for $T\geq 0$.
Next, we have that
\[
\limsup_{t\to\infty} \frac{Y^\ast(t)}{\sqrt{2\log t}}
=\limsup_{t\to\infty} \frac{|Y(t)|_2}{\sqrt{2\log t}}.
\]
We already know from Theorem~\ref{finidimneut} that there is a
$c_0>0$ such that
\[
\limsup_{t\to\infty} \frac{|Y(t)|_\infty}{\sqrt{2\log t}}=c_0,\quad
\text{a.s.}
\]
so by norm-equivalence there is a deterministic $c_1>0$ such that
\[
\limsup_{t\to\infty} \frac{|Y(t)|_2}{\sqrt{2\log t}}\leq c_1,\quad
\text{a.s.}
\]
Hence
\[
\limsup_{t\to\infty} \frac{Y^\ast(t)}{\sqrt{2\log t}}\leq c_1,\quad
\text{a.s.}
\]
Let $\Omega^\ast$ be the event for which this holds. Then by
\eqref{eq.ZastFinal}, for each $\omega\in \Omega^\ast$ we have
\begin{equation*}
\limsup_{T\to\infty} \frac{Z^\ast(T,\omega)}{\sqrt{2\log T}} \leq
2c\varepsilon \limsup_{T\to\infty} \frac{Y^\ast(T)}{\sqrt{2\log T}}
\leq 2c_1c\varepsilon.
\end{equation*}
Since $\varepsilon>0$ is arbitrary, we have
\begin{equation*}
\limsup_{T\to\infty} \frac{Z^\ast(T,\omega)}{\sqrt{2\log T}} = 0,
\quad \text{for each $\omega\in \Omega^\ast$}.
\end{equation*}
Since $\Omega^\ast$ is an almost sure event we have
\begin{equation*}
\limsup_{T\to\infty} \frac{Z^\ast(T)}{\sqrt{2\log T}} = 0, \quad
\text{a.s.}
\end{equation*}
That is
\[
\lim_{t\to\infty} \frac{|X(t)-Y(t)|_2}{\sqrt{2\log t}}=0,
\quad\text{a.s.},
\]
and moreover
\begin{equation}\label{XYlimitneut}
\lim_{t\to\infty} \frac{|X_i(t)-Y_i(t)|}{\sqrt{2\log t}}=0, \quad
\text{a.s.}
\end{equation}
Now, it is known from Theorem~\ref{finidimneut} that
\[
\limsup_{t\to\infty} \frac{|Y_i(t)|}{\sqrt{2\log t}}=\sigma_i,
\quad\text{a.s.}
\]
where $\sigma_i$ is given by \eqref{eq.sigmaineut}. Thus
\[
\limsup_{t\to\infty} \frac{|X_i(t)|}{\sqrt{2\log t}}\leq
\limsup_{t\to\infty} \frac{|Y_i(t)|}{\sqrt{2\log t}}
+\limsup_{t\to\infty} \frac{|X_i(t)-Y_i(t)|}{\sqrt{2\log
t}}=\sigma_i, \quad \text{a.s.}
\]
Similarly
\begin{align*}
\limsup_{t\to\infty} \frac{|X_i(t)|}{\sqrt{2\log t}} &\geq
\limsup_{t\to\infty} \left(\frac{|Y_i(t)|}{\sqrt{2\log t}}-\frac{|X_i(t)-Y_i(t)|}{\sqrt{2\log t}}\right)\\
&= \limsup_{t\to\infty} \frac{|Y_i(t)|}{\sqrt{2\log t}} =\sigma_i,
\quad\text{a.s.}
\end{align*}
Combining these inequalities, we get
\[
\limsup_{t\to\infty} \frac{|X_i(t)|}{\sqrt{2\log t}}= \sigma_i,
\quad\text{a.s.}
\]
Write $X_i(t)=X_i(t)-Y_i(t)+Y_i(t)$. By Theorem~\ref{finidimneut} we have 
\[
\limsup_{t\to\infty}\frac{Y_i(t)}{\sqrt{2\log{t}}}=\sigma_i,\quad\text{a.s.}
\]
Therefore, taking this in conjunction with \eqref{XYlimitneut} we get 
\[
\limsup_{t\to\infty} \frac{X_i(t)}{\sqrt{2\log t}}=\lim_{t\to\infty}\frac{X_i(t)-Y_i(t)}{\sqrt{2\log t}}+
\limsup_{t\to\infty} \frac{Y_i(t)}{\sqrt{2\log t}}=\sigma_i, \quad \text{a.s.},
\]
which is the first part of \eqref{eq:finidimlimsupcompnon}. 

Similarly, since Theorem~\ref{finidimneut} implies 
\[
\limsup_{t\to\infty} \frac{-Y_i(t)}{\sqrt{2\log t}}=\sigma_i, \quad \text{a.s.}
\]
by using this in conjunction with \eqref{XYlimitneut} we have
\[
\limsup_{t\to\infty} \frac{-X_i(t)}{\sqrt{2\log t}}
=\limsup_{t\to\infty} \left(\frac{-Y_i(t)}{\sqrt{2\log t}}+\frac{Y_i(t)-X_i(t)}{\sqrt{2\log t}}\right)
=\sigma_i, \quad \text{a.s.}
\]
This implies 
\[
\liminf_{t\to\infty} \frac{X_i(t)}{\sqrt{2\log t}}=-\sigma_i, \quad \text{a.s.},
\]
which is the second part of \eqref{eq:finidimlimsupcompnon}. We may proceed as in the proof of
Theorem~\ref{finidimneut} to show that these limits imply
\[
\limsup_{t\to\infty} \frac{|X(t)|_\infty}{\sqrt{2\log
t}}=\max_{1\leq i\leq d} \sigma_i, \quad\text{a.s.},
\]
proving the result.


\begin{thebibliography}{18}
%
%

\bibitem{ApplebyMaoWu09neuexi}
J.~A.~D.~Appleby, X.~Mao and H.~Wu, On the existence and uniqueness
of solutions to stochastic equations of neutral type, in
preparation.


\bibitem{AppleWu08linearsfde}
J.~A.~D.~Appleby, X.~Mao and H.~Wu, On the almost sure running
maxima of solutions of affine stochastic functional differential
equations, \emph{SIAM J. Math. Anal.}, 42(2), 646--678, 2010. 

\bibitem{BakBochPaul:1998}
C.~T.~H.~Baker, G.~A.~Bocharov, C.~A.~H.~Paul and F.~A.~Rihan,
Modelling and analysis of time-lags in some basic patterns of cell
proliferation,  \emph{J. Math. Biol.}, 37 (4), 341--371, 1998.

\bibitem{BalJansMcClint:2003}
A.~G.~Balanov, N.~B.~Janson, P.~V.~E.~McClintock, R.~W.~Tucker, and
C.~H.~T.~Wang, Bifurcation analysis of a neutral delay differential
equation modelling the torsional motion of a driven drill-string,
\emph{Chaos, Solitons and Fractals}, 15 (2), 381--394, 2003.


\bibitem{Chukwu:1994}
E.~N.~Chukwu, Control in $W_2^1$ of nonlinear interconnected systems
of neutral type, \emph{J. Austral. Math. Soc. Ser. B}, 36, 286--312,
1994.

\bibitem{ChukwuSimp:1989}
E.~N.~Chukwu and H.~C.~Simpson, Perturbations of nonlinear systems
of neutral type, \emph{J. Differential
Equations} 82, 28-59, 1989.  

\bibitem{CourGlassKeen:1993}
M.~Courtemanche, L.~Glass and J.~P.~Keener, \emph{Phys. Rev. Lett.} 70 2182--???, 1993.

\bibitem{Dieetal95}
O. Diekmann, S.A. van Gils, S.M. Verduyn Lunel and H.-O. Walther,
\emph{Delay equations. Functional-, Complex-, and Nonlinear
Analysis}, New York: Springer-Verlag, 1995.

\bibitem{Frank:2005a} T.~D.~Frank, Stationary distributions of stochastic processes
described by a linear neutral delay differential equation, \emph{J.
Phys. A: Math. Gen.} 38, 485–-490, 2005.

\bibitem{Govindin:2005}
T.~E.~Govindan, Almost sure exponential stability for stochastic
neutral partial functional differential equations,
\emph{Stochastics}, 77, 139–-154, 2005.

\bibitem{GripenbergLonden90}
G.~Gripenberg, S.-O. Londen, and O.~Staffans, \emph{Volterra
integral and functional equations}, Cambridge University Press,
Cambridge, 1990.

\bibitem{Hale:71}
J.~Hale, Forward and backward continuation of neutral functional differential equations, \emph{J. Differential Equations}, 9, 168--181, 1971.  

\bibitem{Hale:77}
J.~Hale, Theory of functional differential equations, Springer-Verlag, New York, 1977. 

\bibitem{HaleLun93}
J.~K.~Hale and S.~M.~Verduyn Lunel, \emph{Introduction to Functional
Differential Equations}, New York: Springer-Verlag, 1993.

\bibitem{HaleCruz:70}
J.~Hale and M.~A.~Cruz, Existence, uniqueness and continuous
dependence for hereditary systems,
\emph{Ann. Mat. Pura Appl.}, (4) 85, 63--82, 1970.  

%

\bibitem{KolNosov:1981}
V.~B.~Kolmanovskii and V.~R.~Nosov, Stability and Periodic Modes of
Control Systems with Aftereffect, Nauka, Moscow, 1981.

\bibitem{KolNosov:1986}
V.~B.~Kolmanovskii and V.~R.~Nosov, Stability of functional-differential equations. Academic Press, London, 1986.

\bibitem{KolMys:99}
V.~B.~Kolmanovskii and A.~D.~Myskhis, Introduction to the theory and applications of functional differential equations,
Kluwer Academic, Dordrecht, 1999.

\bibitem{Luo:2007} J.~Luo,
Fixed points and stability of neutral stochastic delay differential
equation, \emph{J. Math. Anal. Appl.}, 334 (1), 431--440, 2007

\bibitem{Liu&Xia99}
K.~Liu and X.~Xia, On the exponential stability in mean square of
neutral stochastic functional differential equations, \emph{J
Systems Control Lett.}, 37 (4), 207--215,1999.

\bibitem{Luo:2009}
J.~Luo, Exponential stability for stochastic neutral partial
functional differential equations \emph{J. Math. Anal. Appl.}, 355
(1), 414--425, 2009.

\bibitem{LuoMaoShen:2006}
Q.~Luo, X.~Mao and Y.~Shen, New criteria on exponential stability of
neutral stochastic differential delay equations, \emph{Systems and
Control Letters}, 55 (10), 826--834, 2006.

%

\bibitem{Mao97}
X.~Mao, \emph{Stochastic differential equations and applications},
Horwood, 1997.

\bibitem{Mao2008} X.~Mao, \emph{Stochastic differential equations and
applications}, Horwood, Chichester, Second Edition, 2008 [Chapter
6].

\bibitem{Mao:1995}
X.~Mao, Exponential stability in mean square of neutral stochastic
differential functional equations, \emph{Systems Control Lett.}, 26,
245--251, 1995.

\bibitem{Mao97neut} X.~Mao, Razumikhin--type theorems on exponential stability of neutral stochastic functional differential
 equations, \emph{SIAM J. Math. Anal.} 28(2), 389--401, 1997.

\bibitem{Mao:2000}
X.~Mao, Asymptotic properties of neutral stochastic differential
delay equations, \emph{Stochastics and stochastics reports}, 68:3,
273--295, 2000.

\bibitem{LiaoMao:1996}
X.~X.~Liao and X.~Mao, Almost Sure Exponential Stability of Neutral
Differential Difference Equations with Damped Stochastic
Perturbations, \emph{Electro. J. Probab.}, 1, Paper no. 8, 1--16,
1996

\bibitem{MR:2007}
X.~Mao and M.~J.~Rassias, Almost sure asymptotic estimations for
solutions of stochastic differential delay equations, \emph{Int. J.
Appl. Math. Stat.}, No. J07 9, 95--109, 2007.

\bibitem{MaoShenYuan}
X.~Mao, Y.~Shen and C.~Yuan, Almost surely asymptotic stability of
neutral stochastic differential delay equations with Markovian
switching, \emph{Stoch. Proc. Appl.} (to appear).

\bibitem{MaoWu:2008}
X.~Mao, F.~Wu, Numerical Solutions of Neutral Stochastic Functional
Differential Equations \emph{SIAM Journal on Numerical Analysis}, 46
(4), 1821--1841, 2008.

\bibitem{Mo84}
S.-E. A.~Mohammed, Stochastic functional differential equations.
Research Notes in Mathematics, 99, Pitman, Boston, MA, 1984.

\bibitem{protterbook}
P.~E.~Protter, Stochastic integration and differential equations. Second edition. Springer, New York,  2004.


\bibitem{RandJank:2007}
J.~Randjelovic, S.~Jankovic, On the $p$--th moment exponential
stability criteria of neutral stochastic functional differential
equations, \emph{J. Math. Anal. Appl.}, 326 (1), 266--280, 2007.

\bibitem{JankRandJov:2009}
S.~Jankovic, J.~Randjelovic, and M.~Jovanovic, Razumikhin-type
exponential stability criteria of neutral stochastic functional
differential equations, \emph{J. Math. Anal. Appl.}, 355 (2),
811--820, 2009.

\bibitem{ReRiGa06}
M. Rei{\ss}, M. Riedle and O. van Gaans, On Emery's inequality and a
variation-of-constants formula, \emph{Stochastic Anal. Appl.}, 25
(2), 353--379, 2007.

\bibitem{ShenLiao:1999}
Y.~Shen  and X.~Liao, Razumikhin-type theorems on exponential
stability of neutral stochastic functional differential equations,
\emph{Journal Chinese Science Bulletin}, 44 (24), 2225-2228, 1999.

\bibitem{Stepan:1989}
G.~Stepan, Retarded Dynamical Systems: Stability and Characteristic Functions, Longman, New York, 1989.

\bibitem{WuXia96}
J. Wu and H. Xia, Self-sustained oscillations in a ring array of
coupled lossless transmission lines, \emph{J. Differential
Equations.} 124 (1), 247--278, 1996.

\end{thebibliography}
\end{document}